\documentclass[reqno,12pt]{amsart}
\usepackage{amsfonts}
\usepackage{amsmath,graphicx,amssymb,amsthm}
\usepackage{amsmath}
\usepackage{amssymb}
\usepackage{color}
\usepackage{graphicx}
\usepackage{comment}
\usepackage{mathrsfs}  
\usepackage{tikz-cd}
\usepackage[shortlabels]{enumitem}

\let\emptyset\varnothing

\providecommand{\U}[1]{\protect\rule{.1in}{.1in}}
\usepackage[a4paper, total={6in, 8in}]{geometry}
\theoremstyle{definition}
\newtheorem{set}{set}[section]
\newtheorem{theorem}[set]{Theorem}

\newtheorem{corollary}[set]{Corollary}

\newtheorem{definition}[set]{Definition}
\newtheorem{example}[set]{Example}

\newtheorem{lemma}[set]{Lemma}

\newtheorem{observation}[set]{Observation}

\newtheorem{proposition}[set]{Proposition}

\newtheorem{refer}[set]{Citation}

\newtheorem{question}[set]{Question}

\DeclareMathOperator{\II}{II}

\newcommand{\h}{\mathcal{H}}
\newcommand{\image}{\text{Im }}

\newcommand{\norm}[1]{\left\Vert#1\right\Vert}
\newcommand{\aut}{\text{Aut}}

\newcommand{\eb}[1]{\textcolor{blue}{#1}}

\allowdisplaybreaks[4]

\begin{document}
\title[von Neumann Algebras of Thompson-like Groups II]{von Neumann Algebras of Thompson-like Groups from Cloning Systems II}


\author{Eli Bashwinger}
\address{Department of Mathematics and Statistics, University at Albany (SUNY), Albany, NY 12222}
\email{ebashwinger@albany.edu}

\maketitle

\begin{abstract}
Let $(G_n)_{n \in \mathbb{N}}$ be a sequence of groups equipped with a $d$-ary cloning system and denote by $\mathscr{T}_d(G_*)$ the resulting Thompson-like group. In previous work joint with Zaremsky, we obtained structural results concerning the group von Neumann algebra of $\mathscr{T}_d(G_*)$, denoted by $L(\mathscr{T}_d(G_*))$. Under some natural assumptions on the $d$-ary cloning system, we proved that $L(\mathscr{T}_d(G_*))$ is a type $\II_1$ factor. With a few additional natural assumptions, we proved that $L(\mathscr{T}_d(G_*))$ is, moreover, a McDuff factor. In this paper, we further analyze the structure of $L(\mathscr{T}_d(G_*))$, in particular the inclusion $L(F_d) \subseteq L(\mathscr{T}_d(G_*))$, where $F_d$ is the smallest of the Higman--Thompson groups. We prove that if the $d$-ary cloning system is ``diverse," then $L(F_d) \subseteq L(\mathscr{T}_d(G_*))$ satisfies the weak asymptotic homomorphism property. As a consequence, the inclusion is irreducible, which is a considerable improvement of our result that $L(\mathscr{T}_d(G_*))$ is a type $\II_1$ factor, and the inclusion is also singular. Then we look at examples of non-diverse $d$-ary cloning systems with respect to the weak asymptotic homomorphism property, singularity, and irreducibility. Then we finish the paper with some applications. We construct a machine which takes in an arbitrary group and finite group and produces an inclusion (both finite and infinite index) of type $\II_1$ factors which is singular but without the weak asymptotic homomorphism property. Finally,  using irreducibility of the inclusion $L(F_d) \subseteq L(\mathscr{T}_d(G_*))$, our conditions for when $L(\mathscr{T}_d(G_*))$ is a McDuff factor, and the fact that Higman-Thompson groups $F_d$ are character rigid (in the sense of Peterson), we prove that the groups $F_d$ are McDuff (in the sense of Deprez-Vaes).
\end{abstract}

\section*{Introduction}

In \cite{witzel18}, the concept of a cloning system was developed by Stefan Witzel and Matthew Zaremsky to furnish us with a systematic and principled way of constructing generalized Thompson groups based on the classical groups $F \le T \le V$ of Richard Thompson, who introduced them in his work on logic in the 1960s, in addition to allowing us to view the already existing generalizations in a more unified way (see, e.g., \cite{mckenzie73} and see \cite{cannon96} for an introduction to these groups, and also see the survey \cite{zaremsky18user} on cloning systems). The groups $F \le T \le V$ have many guises but one particularly nice viewpoint, which was the impetus for developing cloning systems, is to think of the elements of $V$ as certain equivalence classes of triples. These triples are of the form $(T_-, \sigma , T_+)$ where $T_-$ and $T_+$ are finite rooted binary trees with, say, $n$ leaves labelled left to right by $\{1,\dots,n\}$ and $\sigma$ is a permutation in $S_n$, thought of as permuting the leaf numbering. We denote the equivalence class of the triple $(T_-, \sigma , T_+)$ by $[T_-,\sigma, T_+]$, and the set of equivalence classes of these triples comes equipped with a certain natural binary operation, both of which we will explain later, whereby $V$ becomes a group. Elements of $T$ are those triples where the permutation $\sigma$ is in $\langle (1~2~\cdots ~n) \rangle$, the subgroup of $S_n$ generated by the $n$-cycle $(1~2~\cdots~n)$, and elements of $F$ are those triples where $\sigma$ is always the identity permutation.

Shortly after this paper, in \cite{skipper21} Skipper and Zaremsky generalized this to so-called $d$-ary cloning systems for $d \ge 2$  with $d$ representing the ``arity" of the trees in the triple (binary trees, ternery trees, etc.). The case where $d=2$ corresponds to the original cloning system construction developed by Witzel and Zaremsky, which one might regard as the ``classical" case. One reason cloning systems were expanded to the general $d$-ary case was to add the R{\"o}ver--Nekrashevych groups (see Section \ref{nekrashevych}), first fully introduced in \cite{nekrashevych04}, to the fold of Thompson-like groups arising from cloning systems. The $d$-ary cloning system construction produces generalizations of the Higman--Thompson groups $F_d \le T_d \le V_d$, which are themselves very natural generalizations of the classical Thompson's groups $F \le T \le V$.\footnote{We note that, technically, the groups $F_d$ and $T_d$ were not studied by Higman, and in fact first appeared in work of Brown \cite{brown87}, but to quote Brown, ``\textit{they are simply the obvious generalizations of Thompson's $F$ and $T$.}" Hence, we refer to $F_d$, $T_d$, and $V_d$ collectively as Higman--Thompson groups without fear of reproach. Also, the groups we denote here by $F_d$, $T_d$, and $V_d$ were denoted $F_{d,\infty}$, $T_{d,1}$, and $G_{d,1}$, respectively, in \cite{brown87}.} Indeed, elements of $V_d$ can be thought of as certain equivalence classes of triples $(T_-,\sigma, T_+)$ where in this case $T_-$ and $T_+$ are finite rooted $d$-ary trees and $\sigma$ is still a permutation, and $T_d$ and $F_d$ are defined in a similar fashion to the $d=2$ case. 

The Thompson-like group that results from a $d$-ary cloning system is constructed analogously to the way the Higman--Thompson groups $F_d \le T_d \le V_d$ are constructed. Given a sequence of groups $(G_n)_{n \in \mathbb{N}}$ where there are certain injective functions (not necessarily homomorphisms) between the groups $G_n \to G_{n+d-1}$ and an action of each $G_n$ on $\{1,\dots,n\}$ for every $n \in \mathbb{N}$, we can construct a Thompson-like group, denoted by $\mathscr{T}_d(G_*)$, which canonically contains the Higman--Thompson group $F_d$. In this setup, elements of $\mathscr{T}_d(G_*)$ are, as in the case of the Higman--Thompson groups, certain equivalence classes of triples with the permutation in the middle replaced by a group element from one of the groups $G_n$, which can still be regarded as permuting the leaf numbering since each group $G_n$ comes with an action on $\{1,\dots,n\}$ for every $n \in \mathbb{N}$. 

The Thompson-like group $\mathscr{T}_d(G_*)$ can be thought of as a Thompson-esque limit of the sequence $(G_n)_{n \in \mathbb{N}}$, but this construction ought to be contrasted with the usual direct or inductive limit construction. For example, the Thompson-esque limit preserves finiteness properties better than the direct limit construction. As a matter of fact, the finite symmetric groups $(S_n)_{n \in \mathbb{N}}$ are finitely presented but their usual injective direct limit yields $S_{\infty}$, the group of all finitary permutations of $\mathbb{N}$, which is not even finitely generated. On the other hand, the Thompson-esque limit of $(S_{n})_{n \in \mathbb{N}}$ is the Higman--Thompson group $V_d$ which is finitely presented -- as a matter of fact, it is of type $F_{\infty}$. It is also worth noting that $S_{\infty}$ is amenable while $V_d$ is not amenable, which is a more salient difference from an operator-algebraic perspective. 

As we mentioned above, many of the well-known generalizations fit within the framework of $d$-ary cloning systems (e.g., braided Higman--Thompson groups \cite{dehornoy06,brin07,aroca22}, the generalized Thompson groups of Tanushevski \cite{tanushevski16}, and R{\"o}ver--Nekrashevych groups just to name a few) and many new generalizations have been constructed using $d$-ary cloning systems. Besides simply producing new generalizations of the Higman--Thompson groups, they have since proved useful in a variety of contexts and connect to other areas of mathematics. For example, producing simple groups separated by finiteness properties \cite{skipper19}, inspecting inheritance properties of (bi-)orderability \cite{ishida18}, producing potential counterexamples to the conjecture that every $\text{co}\mathcal{CF}$ group embeds into Thompson's group $V$ \cite{berns-zieve18}, connections to Jones' technology used, for example, to produce certain actions and unitary representations of Thompson's groups, and now in the context of von Neumann algebras for producing intriguing new examples exhibiting a wide range of properties. We note that the connection to Jones' technology is especially intriguing, and we refer the interested reader to \cite{brothier19}, \cite{brothier20}, and \cite{brothier21} for more about Jones' technology and its connection to cloning systems.

In \cite{bashwinger}, we initiated the study of von Neumann algebras of Thompson-like groups arising from $d$-ary cloning systems, and in this paper we continue the analysis of these von Neumann algebras. In that paper, we obtained general structural results concerning the von Neumann algebras of Thompson-like groups which arise from $d$-ary cloning systems. We proved that the preponderance of these groups yield type $\II_1$ factors by virtue of having the infinite conjugacy class (ICC) property,  and even stronger than that, we showed that many of them even yield type $\II_1$ McDuff factors and hence these Thompson-like groups are inner amenable. More precisely, among many other results we proved the following:

\begin{refer}{\cite[Theorem~3.6]{bashwinger}}\label{factor}
Let $((G_n)_{n \in \mathbb{N}}, (\rho_n)_{n \in \mathbb{N}},(\kappa_{k}^{n})_{k \le n})$ be a fully compatible $d$-ary cloning system. If all the $G_n$ are ICC, then so is $\mathscr{T}_d(G_*)$. If the $G_n$ are not necessarily ICC but the $d$-ary cloning system is additionally diverse, then $\mathscr{T}_d(G_*)$ is ICC. 
\end{refer}

\begin{refer}{\cite[Theorem~5.10]{bashwinger}}\label{mcduff}
Let $((G_n)_{n \in \mathbb{N}}, (\rho_n)_{n \in \mathbb{N}},(\kappa_{k}^{n})_{k \le n})$ be a fully compatible, slightly pure, uniform $d$-ary cloning system. Assume that either all the $G_n$ are ICC, or that the $d$-ary cloning system is diverse (so in either case $\mathscr{T}_d(G_*)$ is ICC). Then $L(\mathscr{T}_d(G_*))$ is a McDuff factor and $\mathscr{T}_d(G_*)$ is inner amenable. 
\end{refer}

For the definition of fully compatible and slightly pure, see Section \ref{fullycompatible}; for the definition of diverse and uniform, see Section \ref{diverse}. Regarding Citation~\ref{factor}, fully compatible and diverse $d$-ary cloning systems encompass virtually all of the important examples and Citation \ref{factor} says that Thompson-like groups arising from fully compatible and diverse $d$-ary cloning systems have type $\II_1$ factor group von Neumann algebras. The most important examples of Thompson-like groups arising from $d$-ary cloning systems which are almost never fully compatible (and sometimes not even diverse) are the R{\"o}ver--Nekrasyvech groups (see Section \ref{nekrashevych}). But even the group von Neumann algebras of R{\"o}ver--Nekrasyvech groups turn out to be type $\II_1$ factors with a separate argument (see \cite[Proposition 4.7]{bashwinger}). 

As for Citation~\ref{mcduff}, although there are a lot of hypotheses needed in order to ensure that $L(\mathscr{T}_d(G_*))$ is a type $\II_1$ McDuff factor, it turns out these are very natural conditions and many examples of $d$-ary cloning systems satisfy these conditions, especially the ones which motivated the introduction of cloning systems. For example, as a consequence of this theorem, to our surprise we were able to deduce that the braided Higman--Thompson groups $bF_d$ yield $\II_1$ McDuff factors and hence are inner amenable; we were also able to deduce that the groups $\widehat{V}_d$ yield type $\II_1$ McDuff factors and hence are inner amenable (see \cite[Example 5.8]{bashwinger} for the precise definition of $\widehat{V}_d$). What is surprising about both of these results is that the $bF_d$ are replete with free subgroups, and the $\widehat{V}_d$ are quite similar to $V_d$ (indeed, $\widehat{V}_d$ and $V_d$ embed into each other) but the $V_d$ (as well as $T_d$) are known to be non-inner amenable by a result of the author and Zaremsky in \cite{bashwinger1}, which is an improvement and extension, but not generalization, of Haagerup and Olesen's result in \cite{haagerup17} that $T$ and $V$ are non-inner amenable. Using $d$-ary cloning systems, we also constructed a machine for producing type $\II_1$ McDuff factors from any arbitrary countable group $G$ and hence a machine producing inner amenable groups (see \cite[Example 5.12]{bashwinger}). Later in this paper (see Section \ref{singularnotWAHP}) we slightly modify this machine to construct another machine which takes in an arbitrary countable group and finite group and produces a singular inclusion of type $\II_1$ factors without the weak asymptotic homomorphism property, where the index of the inclusion can be taken to be either finite or infinite. We refer the reader to \cite{bashwinger} for many other intriguing examples of type $\II_1$ factors and type $\II_1$ McDuff factors arising from $d$-ary cloning systems.

The fact that there are Thompson-like groups which yield type $\II_1$ McDuff factors contrasts quite starkly with the groups usually studied in geometric group theory. For example, a large class of groups studied in geometric group theory are acylindrically hyperbolic, and Dahmani–Guirardel–Osin proved in \cite{dahmani17} that acylindrically hyperbolic ICC groups cannot be inner amenable and hence cannot yield type $\II_1$ McDuff factors. This entails that groups yielding type $\II_1$ McDuff factors are in a sense quite rare in geometric group theory. Hence, Thompson-like groups are rather peculiar in this regard, which further adds to their already existing curiosity.

In this paper, we prove, among many other things, a result which considerably strengthens Citation~\ref{factor}. We prove a structural result about how the Higman-Thompson group factor $L(F_d)$ sits inside Thompson-like group factors arising from $d$-ary cloning systems. To be more specific, in Section \ref{WAHP} we prove Theorem \ref{thrm:WAHP} which states that that if a sequence of groups $(G_n)_{n \in \mathbb{N}}$ is equipped with a diverse $d$-ary cloning system, then the inclusion $L(F_d) \subseteq L(\mathscr{T}_d(G_*))$ satisfies the weak asymptotic homomorphism property (see Definition \ref{def:WAHP}), and this will have a multitude of consequences. For example, this also proves that the inclusion $L(F_d) \subseteq L(\mathscr{T}_d(G_*))$ is singular and that it is irreducible (see Section \ref{WAHP}). Furthermore, the fact that the inclusion $L(F_d) \subseteq L(\mathscr{T}_d(G_*))$ is irreducible will itself have a number of consequences. The fact that the inclusion $L(F_d) \subseteq L(\mathscr{T}_d(G_*))$ is irreducible represents a considerable improvement of Citation \ref{factor}. This simultaneously strengthens Citation~\ref{factor}, as we can dispense with the fully compatible assumption and prove the stronger conclusion of irreducibility, and it further reveals the structure of $L(\mathscr{T}_d(G_*))$. The beauty of this result also lies in the fact that we no longer need to rely on results from \cite{preaux13} about when a group extension yields an ICC group, which were used to prove Citation~\ref{factor}. The proof of Theorem \ref{thrm:WAHP} is quite self-contained, albeit somewhat technical.

In addition to these consequences of the inclusion $L(F_d) \subseteq L(\mathscr{T}_d(G_*))$ satisfying the weak asymptotic homomorphism property, the other significance of this result is that it relates these examples to work of Popa. Somewhat more precisely, it relates it to Popa's intertwining-by-bimodules technique in the contemporary structure theory of $\II_{1}$ factors which is analogous to structure-randomness dichotomies of ergodic dynamical systems. In particular, the corner embeddability condition appearing in Popa's technique is formally similar to a weak mixing condition. In fact, singularity of maximal abelian subalgebras can alternatively be characterized by weak mixing conditions or the weak asymptotic homomorphism property. The situation is subtler for nonabelian subalgebras. In this more general setting, the weak asymptotic homomorphism property implies singularity, but the converse does not necessarily hold the other way around (see \cite{grossman10}). For more background on Popa's intertwining-by-bimodules technique, we refer the interested reader to \cite[Lemmas 4 and 5]{Popa04}, \cite[Theorem A.1]{popa06betti}, and \cite[Section 2]{popa06rigidity}

In Section \ref{nekrashevych}, we treat the R{\"o}ver--Nekrashevych groups separately with respect to irreducibility, singularity, and the weak asymptotic homomorphism property. The reason for treating the these groups separately is that, although every R{\"o}ver--Nekrashevych group arises from a $d$-ary cloning system, it is not always the case that these $d$-ary cloning systems are diverse. Hence, we cannot use any of our current theorems to conclude, for example, that $L(F_d)$ is an irreducible subfactor in the R{\"o}ver--Nekrashevych group factors. Interestingly, however, for these groups it turns out we can prove something stronger by treating them separately, namely, that $L([F_d,F_d])$ being an irreducible subfactor of the R{\"o}ver--Nekrashevych group factors. Of course, this entails that $L(F_d)$ is also an irreducible subfactor of the R{\"o}ver--Nekrashevych group factors. However, it is not always that case that the inclusion is singular and hence does not satisfies the weak asymptotic homomorphism property. Some of the R{\"o}ver--Nekrashevych groups do arise from diverse $d$-ary cloning systems, and obviously for these the inclusion satisfies the weak asymptotic homomorphism property, although we do not know how to completely characterize such $d$-ary cloning systems at the moment. In Section \ref{nekrashevych}, we provide a necessary condition for the $d$-ary cloning system to be diverse, but we do not know if it is sufficient.

Recall that for an inclusion of tracial von Neumann algebras $N \subseteq M$, satisfying the weak asymptotic homomorphism property is equivalent to $N$ being a weakly mixing von Neumann subalgebra (see Section \ref{normalizers}). Given how strong the diversity assumption is, one might naturally wonder whether we can prove that $L(F_d)$ has the stronger property of being a mixing subfactor of $L(\mathscr{T}_d(G_*))$ whenever $(G_n)_{n \in \mathbb{N}}$ is equipped with a diverse $d$-ary cloning system. In Section \ref{mixing}, we will see that mixing is almost never possible. Hence, for essentially all the most important examples, $L(F_d)$ will be a weakly mixing subfactor of $L(\mathscr{T}_d(G_*))$ which is not mixing. Although mixing of $L(F_d)$ cannot be entirely ruled out at the moment, we do not know how to construct a $d$-ary cloning system on a sequence of groups $(G_n)_{n \in \mathbb{N}}$ such that $L(F_d)$ is a mixing subfactor of $L(\mathscr{T}_d(G_*))$. 

In Section \ref{non-diverse}, we look at examples of Thompson-like groups which arise from non-diverse $d$-ary cloning systems and investigate whether or not $L(F_d) \subseteq L(\mathscr{T}_d(G_*))$ satisfies the weak asymptotic homomorphism property. We show that both are possible with non-diverse $d$-ary cloning systems. On a related note, we also wonder whether it is possible to construct a $d$-ary cloning
system (necessarily non-diverse) with resulting group $\mathscr{T}_d(G_*)$ such that $L(F_d) \subseteq L(\mathscr{T}_d(G_*))$ is singular inclusion of type $\II_1$ factors but does not satisfy the weak asymptotic homomorphism property. At the moment, however, this seems rather difficult. 

Despite the fact that we cannot (yet) construct a $d$-ary cloning system on a sequence of groups $(G_n)_{n \in \mathbb{N}}$ such that $L(F_d) \subseteq L(\mathscr{T}_d(G_*))$ is a singular inclusion of type $\II_1$ factors without the weak asymptotic homomorphism property, we can, however, construct inclusions of different Thompson-like groups which yield singular inclusion of type $\II_1$ factors without the weak asymptotic homomorphism property. As we mentioned above, in Section \ref{singularnotWAHP} we construct a machine using $d$-ary cloning systems and the amalgamated free product construction for groups which takes in an any finite group and any countable group and produces a singular inclusion of type $\II_1$ factor without the weak asymptotic homomorphism property with the inclusion being of either finite or infinite index. The examples Grossman and Wiggins constructed in \cite{grossman10} are finite index; in fact, they show more generally that proper finite index inclusions cannot satisfy the weak asymptotic homomorphism property. As far as we can tell, this left open the case of finding an infinite index, singular inclusion of type $\II_1$ factors without the weak asymptotic homomorphism property, which our construction provides. We note that this construction was done concurrently with the ones done in \cite{bannon23}, although ours is different in that it utilizes $d$-ary cloning systems and is purely group-theoretic. Hence, these examples are the first of their kind. 

Finally, in Section \ref{sec:higmanmcduff}, we finish the paper with an application. We use irreducibility of the inclusion $L(F_d) \subseteq L(\mathscr{T}_d(G_*))$, Citation~\ref{mcduff}, and the fact that the Higman--Thompson groups $F_d$ are character rigid in the sense of Peterson to show that the groups $F_d$ are McDuff in the sense of Deprez-Vaes for all $d \ge 2$. In proving the Higman--Thompson groups $F_d$ are McDuff, we will also see that irreduciblity and character rigidity can be used to prove that if a sequence $(G_n)_{n \in \mathbb{N}}$ of (non-trivial) abelian groups is equipped with a pure and diverse cloning system, then a certain canonical subgroup $\mathscr{K}_d(G_*)$ (defined in Section \ref{fullycompatible}) yields a Cartan subalgebra in $L(\mathscr{T}_d(G_*))$. What makes this result somewhat curious is that the Higman--Thompson group factors $L(F_d)$, $L(T_d)$, and $L(V_d)$ cannot contain a Cartan subalgebra arising from an abelian subgroup. Indeed, if $H \le G$ is an inclusion of countable groups, then $L(H)$ is a Cartan subalgebra of $L(G)$ if and only if $H$ is a normal abelian subgroup of $G$ such that $\{h^{-1} g h : h \in H \}$ is infinite for every $g \in G \setminus H$. Hence, in order for $L(H)$ to yield a Cartan subalgebra, minimally, the subgroup has to be normal and abelian. Any normal subgroup of $F_d$ necessarily contains $[F_d,F_d]$ and is therefore non-abelian. As for the Higman--Thompson group factors $L(T_d)$ and $L(V_d)$, the story is somewhat similar. When $d$ is even, both $T_d$ and $V_d$ are simple so, a fortori, they have no normal abelian subgroups; whereas when $d$ is odd, their respective commutator subgroups are simple and of index $2$ (i.e., $T_d$ and $V_d$ are virtually simple), and in this case any infinite, normal subgroup of either $T_d$ or $V_d$ must contain the commutator subgroup and hence be non-abelian. This shows that none of the Higman--Thompson groups contain a subgroup giving rise to a Cartan subalgebra, yet, as we shall see, it is rather easy to produce Thompson-like groups with a subgroup giving rise to a Cartan subalgebra.

\subsection*{Acknowledgements} The author is indebted to Jon Bannon and Matthew Zaremsky for helpful discussions: to Jon Bannon for his help on the general theory of von Neumann algebras and recommending various avenues to explore; and to Matthew Zaremsky for his input on some of the technical aspects of $d$-ary cloning systems and helping tidy up some of the proofs. 

\section{Brief Detour into von Neumann Algebras}\label{detour}

\subsection{Basic Theory and Constructions}\label{basics}

Although this section is intended to be a relatively self-contained treatment of von Neumann algebras, for a excellent general reference overview we refer to Popa's ICM survey (see \cite{popa07}), and for a more detailed introduction we refer to his book \cite{anan} with Anantharaman.

A von Neumann algebra is a $*$-subalgebra of bounded linear operators on some Hilbert space which is closed in the strong (equivalently, weak) operator topology. Arguably, the most well-studied and beloved von Neumann algebras arise from a group together with a unitary representation (e.g., the left-regular representation, the most natural unitary representation one can single out) or groups acting on measure spaces or other von Neumann algebras. When we have a group acting on a von Neumann algebra, we can form the so-called crossed product von Neumann algebra. Let us recall the basics of this construction. Throughout this paper, $G$ will almost exclusively denote a countable, discrete group. An \emph{action} of $G$ on a von Neumann algebra $M$, which we will assume acts on a separable Hilbert space $\h$, is a homomorphism $\sigma : G \to \aut (M)$ of $G$ to the group $\aut (M)$ of (normal) $*$-automorphisms of $M$. Consider the Hilbert space $\ell^2(G, \h)$ defined as follow: $$\ell^2(G,\h) := \{\psi : G \to \h : \sum_{g \in G} \norm{\psi(g)}_{\h}^2 < \infty \}$$ 
equipped with the inner product $$\langle \psi, \phi \rangle_{\ell^2(G,\mathcal{H})} := \sum_{g \in G} \langle \psi (g), \phi(g) \rangle_{\mathcal{H}}$$
where the inner product $\langle , \rangle_{\mathcal{H}}$ on the right-hand side comes from the Hilbert space $\h$, and $\norm{~}_{\h}$ denotes the norm on $\h$ induced by this inner product. Define $\lambda : G \to B(\ell^2(G,\h))$ and $\pi : M \to B(\ell^2(G,\h))$ by 
\begin{align*}
(\lambda (g) \psi )(h) &= \psi (g^{-1}h) \\
(\pi (x) \psi)(h) &= \sigma_{h^{-1}}(x) \psi (h)
\end{align*}
for $g,h \in G$, $x \in M$, and $\psi \in \ell^2(G,\h)$. One can easily verify that the former is a faithful unitary representation of $G$ while the latter is a faithful normal $*$-homomorphism, and that the commutation relation $$\lambda (g) \pi (x) \lambda (g)^* = \pi (\sigma_{g}(x))$$
must hold, where $x \in M$ and $g \in G$. Putting this together, the crossed product von Neumann algebra of $G$ acting on $M$, denoted by $M \rtimes_{\sigma} G$, is defined to be the von Neumann subalgebra of $B(\ell^2(G,\h))$ generated by $\pi (M)$ and $\pi (G)$; that is, form $M \rtimes_{\sigma} G$ by first forming the $*$-algebra generated by $\pi (M)$ and  $\lambda (G)$ and then taking the weak or stong operator topology closure. We note that the $*$-isomorphism class of $M \rtimes_{\sigma} G$ is independent of the choice of separable Hilbert space $\h$ on which we represent $M$. At times we may omit the action $\sigma$ from the notation when the context is clear and simply write $M \rtimes G$. 

If $M = \mathbb{C}$ with $G$ acting trivially, we obtain the so-called group von Neumann algebra of $G$, denoted as $L(G)$. If our group is given by a semi-direct product $N \rtimes_{\sigma} H$, then when we form the group von Neumann algebra of it, the semi-direct product translates into a crossed product von Neumann algebra. That is, $$L(N \rtimes_{\sigma} H) \cong L(N) \rtimes_{\widehat{\sigma}} H,$$
where $\widehat{\sigma}$ denotes the action of $H$ on $L(N)$ induced from the action $\sigma$ of $H$ on $N$. 

\subsection{Factors}\label{factors}

Factor von Neumann algebras, those von Neumann algebras whose center consists only of scalar multiples of the identity, are the simple objects among von Neumann algebras and are therefore expectedly important in the theory. Indeed, all von Neumann algebras can be decomposed as a direct integral of factors. The theory reduces even further to so-called type $\II_1$ factors, meaning they essentially represent the final frontier in the classification of von Neumann algebras. Factors of type $\II_1$ are infinite-dimensional factors which admit a normal, finite, faithful, normalized trace $\tau$. Hence, it is helpful to know when a general procedure for constructing von Neumann algebras, such as the crossed product construction, yields a type $\II_1$ factor.

For crossed product von Neumann algebras, we have some nice standard criteria for when the crossed product is a type $\II_1$ factor in terms of the way by which the group acts. Before we can state these criteria, though, let us recall two important ways a group can act on a von Neumann algebra. If $\sigma$ is an action of group $G$ on a von Neumann algebra $M$, then the action $\sigma$ of $G$ on $M$ is said to be \emph{ergodic} provided $M^{G} = \mathbb{C}1$, where $$M^{G} := \{x \in M : \sigma_{g}(x) = x ~ \forall g \in G\}$$ is the fixed-point subalgebra associated to the action. The action is said to be \emph{free} provided the automorphism $\sigma_g$ is properly outer for every non-trivial element $g \in G$, which means that if $x \in M$ with $xy = \sigma_g(y)x$ for all $y \in M$, it follows that $x=0$. With these definitions in mind, we have the following standard criteria from \cite{anan}:

\begin{refer}\label{crossedproduct}
Let $(M, \tau)$ be a tracial von Neumann algebra and $\sigma$ a trace-preserving action of a group $G$ on $M$. Then:
\begin{enumerate}
    \item $M' \cap (M \rtimes_\sigma G) = Z(M)$ if and only if the action is free.
    \item Assuming that the action is free, then $M \rtimes_\sigma G$ is a factor (and hence a type $\II_1$ factor) if and only if the action of $G$ on $Z(M)$ is ergodic. 
\end{enumerate}
If in addition $G$ is an ICC group, then:
\begin{enumerate}
    \item $L(G)' \cap (M \rtimes_{\sigma} G) = M^G$, and
    \item $M \rtimes G$ is a factor (and hence a type $\II_1$ factor) if and only if the $G$-action on $Z(M)$ is ergodic.
\end{enumerate}
\end{refer}

\subsection{Normalizers, Irreducibility, Singularity, and the Weak Asymptotic Homomorphism Property}\label{normalizers}
Given an inclusion of von Neumann algebras $N \subseteq M$, we can associate a certain group of unitaries in $M$ which stabilize the von Neumann subalgebra $N$ under the conjugation action: $$\mathcal{N}_M(N) := \{u \in \mathcal{U}(M) \mid u^*Nu = N \}$$
Naturally, this is called the \emph{normalizer} of $N$ in $M$, which contains $\mathcal{U}(N)$, the group of all unitary elements in $N$, as a normal subgroup. The normalizer of $N$ generates a von Neumann algebra, first by forming the $*$-algebra generated by the normalizer and then taking the weak (or strong) operator closure or the bicommutant, and the subalgebra $N$ can be classified in terms of the von Neumann algebra its normalizer generates. If the normalizer generates $N$, then $N$ is said to be \emph{singular}; if it generates a proper subalgebra of $M$ other than $N$, then $N$ is said to be \emph{semi-regular}; and, finally, if the normalizer generates $M$, then $N$ is said to be \emph{regular}. 

Now for some terminological conventions. If $H \le G$ is an inclusion of groups, then we say that $H$ is \emph{self-normalizing} in $G$ if the normalizer of $H$ in $G$ is ``trivial;" i.e., if $\mathcal{N}_{G}(H) = H$. Similarly, if $N \subseteq M$ is an inclusion of von Neumann algebras, then we say that $N$ is \emph{self-normalizing} $M$ if the normalizer of $N$ in $M$ is ``trivial;" i.e., if $\mathcal{N}_{M}(N) = \mathcal{U}(N)$.

In addition to looking at normalizers, we can look at one-sided normalizers. The \emph{one-sided normalizer} of $N$ in $M$ is the semigroup defined to be $$\mathcal{ON}_{M}(N) = \{u \in \mathcal{U}(M) : u^* N u \subseteq N \},$$
which contains the usual von Neumann algebra normalizer. Given an inclusion $H \le G$ of groups, we can similarly define the the normalizer and one-sided normalizer of $H$ in $G$ which as denoted as $\mathcal{N}_G(H)$ and $\mathcal{ON}_G(H)$, respectively. 

Let us now turn to the notion of irreducibility.

\begin{definition}[Irreducible]
Let $N \subseteq M$ be an inclusion of von Neumann algebras. We say that the inclusion is \emph{irreducible} if $N' \cap M = Z(N)$.
\end{definition}
Equivalently, $N \subseteq M$ is irreducible if and only if the conjugation action of the unitary group $\mathcal{U}(N)$ on $M$ is ergodic. For this reason, irreducible inclusions are also sometimes called ergodic inclusions or embeddings as in \cite{popa21}.
With this definition of irreduciblity in mind, we can immediately rephrase Citation~\ref{crossedproduct}. First, the inclusion $M \subseteq M \rtimes_\sigma G$ is irreducible if and only if the action is free. Second, if $G$ is an ICC group, then the inclusion $L(G) \subseteq M \rtimes_\sigma G$ is irreducible if and only if the action of $G$ on $M$ is ergodic. 

Irreducible inclusions of von Neumann algebras are nice for a variety of reasons, but for our purposes they are desirable for the following reasons. First, irreducible inclusions are nice because they allow us to ``upgrade" factoriality of a subfactor to the containing or ambient von Neumann algebra. As a matter of fact, it follows quite easily from this definition that if $P$ is an intermediate von Neumann algebra of $N \subseteq M$, then the inclusion $P \subseteq M$ is also irreducible, and, moreover, if $N$ is a factor, then $P$ must also be a factor; in particular, $M$ must be a factor. Second,  when the von Neumann algebras $N \subseteq M$ are group von Neumann algebras arising from an inclusion of groups, the von Neumann algebra normalizer admits a nice description in terms of the group normalizer, and determining whether the von Neumann subalgebra $N$ is regular, semi-regular, or singular is somewhat easier (see Citation \ref{SmithNormalizer}).

It turns out we have some nice group-theoretic characterizations for when the group von Neumann algebra is a factor (and hence a type $\II_1$ factor) and when a group inclusion gives rise to an irreducible inclusion of von Neumann algebras. For factoriality of $L(G)$, it is a standard, classical fact that $L(G)$ is a factor if and only if $G$ is ICC, which means that the conjugacy class of any non-trivial element is infinite. Clearly, the property of being ICC is equivalent to the property that the only element with finite index centralizer is the identity. 
For irreducibility of an inclusion of group von Neumann algebras, we have the following from \cite{smith09} which in a sense generalizes the ICC condition:

\begin{refer}{\cite[Lemma 6.1]{smith09}}\label{SmithIrred}
Let $H \le G$ be an inclusion of countable, discrete groups. Then $L(H)$ is irreducible in $L(G)$ if and only if each $g \in G \setminus \{1\}$ has infinitely many $H$-conjugates, meaning that the set $\{h^{-1}gh : h \in H\}$ is infinite.
\end{refer}

Note that a non-trivial element $g \in G$ having infinitely many $H$-conjugates is equivalent to $g$ having infinite index centralizer in $H$, and when $H=G$ we obtain the usual conditions for when $L(G)$ is a type $\II_1$ factor. 

From \cite{smith09}, we also have a theorem about computing the von Neumann algebra normalizer when we have an irreducible inclusion arising from a group inclusion. This says that the von Neumann algebra normalizer can be directly computed from the group normalizer. 

\begin{refer}{\cite[Theorem~6.2]{smith09}}\label{SmithNormalizer}
Let $H \le G$ be an inclusion of countable discrete groups with $G$ ICC such that the inclusion $L(H) \subseteq L(G)$ is irreducible. Then 
\begin{enumerate}
\item $\mathcal{ON}_{L(G)}(L(H)) = \{ u \lambda_g : u \in \mathcal{U}(L(G)) \text{ and } g \in \mathcal{ON}_G(H)\}$ \\
\item $\mathcal{N}_{L(G)}(L(H)) = \{u \lambda_{g} \mid u \in \mathcal{U}(L(H)) \text{ and } g \in \mathcal{N}_{G}(H)\}$
\end{enumerate}
\end{refer}
Note that if $\lambda : G \to \mathcal{U}(\ell^2(G))$ denotes the left-regular representation of $G$, defined on the canonical basis $\{\delta_x \}_{x \in G }$ of $\ell^2(G)$ as $\lambda(g) \delta_x := \delta_{gx}$ for $g,x \in G$, then $\lambda_g := \lambda(g)$ denotes the unitary element induced by $g \in G$. Roughly, this theorem states that the (one-sided) von Neumann algebra normalizers are made up of the (one-sided) group normalzers modulo a unitary from $L(H)$. Clearly the right-hand side is always contained in the left-hand side, but in the irreducible case we never have a strict inclusion but rather we always have equality. Note, for example, that if $H$ is self-normalizing in $G$, then the inclusion $L(H) \subseteq L(G)$ is singular.  

The last property we discuss in this section is the weak asymptotic homomorphism property. First, given a von Neumann algebra $M$ with a normal, finite, faithful, normalized trace $\tau$, we can form the $2$-norm on $M$ defined by $\norm{x}_2 = \sqrt{\tau(x^*x)}$ for $x \in M$. For simplicity, in the future we will just refer to $\tau$ as a trace. Also, as an aside, we mention that given this norm $\norm{~}_2$ on $M$, we can form the space $\ell^{\infty}(\mathbb{N},M)$ of all $\norm{~}_2$-bounded sequences, which will be useful when stating the McDuff property and property Gamma in Section~\ref{McDuffGammaInnerAmenable}. With this in mind, the weak asymptotic homomorphism property is defined as follows.

\begin{definition}[Weak Asymptotic Homomorphism Property]\label{def:WAHP}
Let $M$ be a von Neumann algebra with trace $\tau$ and let $N$ be a von Neumann subalgebra. Then the inclusion $N \subseteq M$ is said to satisfy the \emph{weak asymptotic homomorphism property} (WAHP) provided there exists a net of unitaries $\{u_i \}_{i \in I}$ in $N$ such that $$\lim_{i} \norm{E_{N}(x u_{i} y) - E_N(x) u_{i} E_{N}(y)}_2 = 0$$
for every $x,y \in M$, where $E_{N
} : M \to N$ denotes the unique, faithful, normal, trace preserving conditional expectation from $M$ onto $N$.  
\end{definition}
As we noted in the introduction, proper finite index inclusions preclude the WAHP, meaning that if $N \subseteq M$ is a proper finite index inclusion, then $N \subseteq M$ cannot satisfy the WAHP (see \cite{grossman10}). The WAHP can also be rephrased in terms of the one-sided quasi-normalizer of $N$ in $M$ being ``trivial," and using this reformulation we can deduce two other easy consequences. The following comes from \cite{fang11}.

\begin{refer}{\cite[Theorem 3.1]{fang11}}\label{WAHPquasi}
Let $M$ be a von Neumann algebra with trace $\tau$ and $N$ a von Neumann subalgebra. Then the following are equivalent:
\begin{enumerate}
    \item The inclusion $N \subseteq M$ satisfies the weak asymptotic homomorphism property.
    \item If $x, x_1,\dots,x_n \in M$ with $Nx \subseteq \sum_{i=1}^{n} x_i N$, then $x \in N$.
\end{enumerate}
\end{refer}

An element $x \in M$ for which there exist $x_1,\dots,x_n \in M$ with $Nx \subseteq \sum_{i=1}^{n} x_i N$ is called a one-sided quasi-normalizer of $N$, so the WAHP is equivalent saying that the one-sided quasi-normalizer of $N$ in $M$ is trivial. We let $\mathcal{QN}^{(1)}_{M}(N)$ denote the set of one-sided quasi-normalizers, and we have the following relationship between the various types of normalizers: $$\mathcal{N}_M(N) \subseteq \mathcal{ON}_M(N) \subseteq \mathcal{QN}_M^{(1)}(N).$$  Using this formulation of the WAHP, we note that it is quite easy to see that if $N \subseteq M$ is an inclusion of von Neumann algebras satisfying the WAHP, then the inclusion must also be singular. More than that, the one-sided 
normalizer of $N$ in $M$ must be trivial (i.e., $\mathcal{ON}_{M}(N) = \mathcal{U}(N)$). Furthermore, it is quite easy to argue that if the inclusion $N \subseteq M$ satisfies the WAHP, then the inclusion must also be irreducible. We note that Theorem~3.1 in \cite{fang11} is actually phrased in terms of the relative weak asymptotic homomorphism property, but Citation~\ref{WAHPquasi} is easily derived from it.  

As Jolissaint notes in \cite{Jolissaint12a}, the inclusion $N \subseteq M$ satisfies the WAHP if and only if the inclusion $N \subseteq M$ is weakly mixing, i.e., $N$ is a weakly mixing subalgebra of $M$ (see \cite[Theorem~1.4]{Jolissaint12a}). This equivalence is also true of the relative versions of the WAHP and weak mixing involving triples of von Neumann algebras, but we do not explore the relative versions in this paper. For more details on the relative WAHP and relative weak mixing, we refer the interested reader to \cite{Jolissaint12a}, \cite{jolissaint12b}, \cite{fang11}, and \cite{cameron13}.

In the case of an inclusion of von Neumann algebras arising from a group inclusion, condition (2) in Citation~\ref{WAHPquasi} translates to a condition about covering a right coset with left cosets. More precisely, we have the following from \cite{fang11}:

\begin{refer}{\cite[Corollary 5.4]{fang11}}\label{groupWAHP}
Let $H \le G$ be an inclusion groups. Then the following are equivalent:
\begin{enumerate}
    \item The inclusion $L(H) \subseteq L(G)$ satisfies the weak asymptotic homomorphism property
    \item If $g,g_1,\dots,g_n \in G$ with $Hg \subseteq \bigcup_{i=1}^{n} g_i H$, then $g \in H$.
\end{enumerate}
\end{refer}

Similar to what we said above, an element $g \in G$ for which there are $g_1, \dots , g_n \in G$ with $Hg \subseteq \bigcup_{i=1}^{n} g_iH$ is called a one-sided quasi-normalizer of $H$. Later in the paper, we may say that an inclusion $H \le G$ of countable discrete groups satisfies the right coset covering condition or that the inclusion has the right coset covering property if condition (2) is satisfy, meaning that the inclusion $L(H) \subseteq L(G)$ satisfies the WAHP if and only if the inclusion $H \le G$ satisfies the right coset covering condition or has the right coset covering property. 

To summarize this subsection, we note that for an inclusion of tracial von Neumann algebras $N \subseteq M$, we have the $N \subseteq M$ satisfying the WAHP implies that it is singular and that it is irreducible. 
In the course of this paper, we will see many examples where all the implications are strict. For an inclusion which is irreducible but does not satisfy the WAHP (or is even singular), see the R{\"o}ver--Nekrashevych groups in Section \ref{nekrashevych}, and for an inclusion which is singular but not does not satisfy the WAHP, see Section \ref{singularnotWAHP}.

\subsection{McDuff Factors, Property Gamma, and Inner Amenability}\label{McDuffGammaInnerAmenable}

Although we are not going to explore the three properties mentioned in the title of this section in relation to $d$-ary cloning systems in this paper as much as we did in \cite{bashwinger}, we do need them to state our results in that paper, particularly Citation~\ref{mcduff}, and then apply said results to prove that the Higman--Thompson groups $F_d$ are McDuff for all $d \ge 2$ in Section \ref{sec:higmanmcduff}. Let us now define the (relative) McDuff property.

\begin{definition}[(Relative) McDuff Property]
A type $\II_1$ factor $M$ is said to be \emph{McDuff} provided $M \cong M \otimes R$, where $R$ is the hyperfinite type $\II_1$ factor. A pair of type $\II_1$ factors $N \subseteq M$ has the \emph{relative McDuff property} provided that there is an isomorphism $M \to M \otimes R$ which restricts to an isomorphism $N \to N \otimes R$. 
\end{definition}

Recall that $R$ can be constructed by forming the group von Neumann algebra of $S_{\infty}$ (or any countable amenable ICC group for that matter), although it has other manifestations. It turns out that the McDuff property can be equivalently phrased in terms of central sequences.

\begin{definition}[Central Sequences]
Let $M$ be a type $\II_1$ factor. A sequence $(a_n)_{n \in \mathbb{N}} \in \ell^{\infty}(\mathbb{N},M)$ is said to be \emph{central} provided
$$\lim_{n \to \infty} \norm{x a_n- a_n x}_2 = 0$$
for all $x \in M$. Two central sequences $(a_n)_{n \in \mathbb{N}}, (b_n)_{n \in \mathbb{N}} \in \ell^{\infty}(\mathbb{N},M)$ are said to be \emph{equivalent} provided 
$$\lim_{n \to \infty} \norm{a_n - b_n}_2 = 0$$
Finally, a central sequence is \emph{trivial} provided it is equivalent to a scalar sequence. 
\end{definition}

Remarkably, the mere existence of a pair of non-trivial, inequivalent central sequences in a type $\II_1$ factor $M$ is equivalent to the existence of an isomorphism $M \to M \otimes R$, i.e., equivalent to $M$ being McDuff (see \cite{mcduff70}).  A related classical invariant of type $\II_1$ factors, which goes back to Murray and von Neumann, the founders of the theory of von Neumann algebras (see \cite{murray43}), is that of property Gamma.

\begin{definition}[Property Gamma]
A type $\II_1$ factor $M$ has \emph{property Gamma} provided there exists a non-trivial central sequence. 
\end{definition}

Phrasing the McDuff property in terms of central sequences, it is clear that the McDuff property entails property Gamma. For an ICC group $G$, a classical result of Effros says that if $L(G)$ has property Gamma, then $G$ is inner amenable (see \cite{effros75}). Recall that a group $G$ is \emph{inner amenable} if it admits a finitely additive, conjugation invariant probability measure on the set of subsets of $G \setminus \{e\}$. In summary, for a countable ICC group $G$ we have the following strict implications:
\[
G \text{ amenable} \Rightarrow L(G) \text{ McDuff} \Rightarrow L(G) \text{ has Property Gamma} \Rightarrow G \text{ inner amenable.}
\]
The first implication is easily seen to be strict: let $G$ be the direct product of any amenable ICC group with $\mathbb{F}_2$, the free group on two generators. As for the second implication being strict, the group von Neumann algebras of some of the Baumslag-Solitar groups have property Gamma (see \cite{Stalder06} and \cite{BannonMO}). However, Fima proved in \cite{Fima11} that they also yield prime factors and hence cannot be McDuff factors. Recall that a von Neumann algebra $M$ is said to be \emph{prime} if whenever $M$ can be decomposed into a tensor product of von Neumann algebras, one of the tensor product factors must be finite dimensional. Finally, Vaes constructed an example in \cite{vaes12} showing that the third implication is strict, an implication which was long thought to be an equivalence. 

\section{Introduction to $d$-ary Cloning Systems}\label{introduction}

\subsection{Basic Axioms and Properties of $d$-ary Cloning Systems}\label{axioms}

In this section, we recall the basics of the $d$-ary cloning system construction which produces Thompson-like groups. We also recall some important properties of $d$-ary cloning systems and point out some examples. We now recall the definition of a $d$-ary cloning system and its axioms.

\begin{definition}
Let $d \ge 2$ be an integer and $(G_n)_{n \in \mathbb{N}}$ a sequence of groups. For each $n \in \mathbb{N}$, let $\rho_n : G_n \to S_n$ be a group homomorphism from $G_n$ to the finite symmetric group $S_n$. For each $1 \le k \le n$, let $\kappa_{k}^{n} : G_n \to G_{n+d-1}$ be an injective function (not necessarily a homomorphism). We write $\rho_n$ to the left of its input and $\kappa_{k}^{n}$ to the right of its input. We call the triple
$$((G_n)_{n \in \mathbb{N}}, (\rho_n)_{n \in \mathbb{N}}, (\kappa_{k}^{n})_{k \le n})$$
a \emph{$d$-ary cloning system} if the following axioms hold:
\begin{enumerate}[(C1)]
    \item (Cloning a product) $(gh)\kappa_{k}^{n} = (g) \kappa_{\rho_n(h)k}^{n} (h) \kappa_{k}^{n}$ 
    \item (Product of clonings) $\kappa_{\ell}^{n} \circ \kappa_{k}^{n+d-1} = \kappa_{k}^{n} \circ \kappa_{\ell + d -1}^{n+d-1}$
    \item (Compatibility) $\rho_{n+d-1}((g)\kappa_{k}^{n})(i) = (\rho_{n}(g)) \varsigma_{k}^{n}(i)$ for all $i \neq k,k+1,\dots,k+d-1$.
\end{enumerate}
Here we always have $1 \le k < \ell \le n$ and $g,h \in G_n$, and the maps $\varsigma_{k}^{n} : S_{n} \to S_{n+d-1}$ denote the so-called standard $d$-ary cloning maps on the finite symmetric groups $(S_n)_{n \in \mathbb{N}}$, which we will momentarily explain. We call the group homomorphism $\rho_n : G_n \to S_n$ a \emph{representation map} for every $n \in \mathbb{N}$ and the injective function $\kappa_k^n : G_n \to G_{n+d-1}$ a \emph{$d$-ary cloning map} for all $1 \le k \le n$.
\end{definition}

Given a $d$-ary cloning system $((G_n)_{n \in \mathbb{N}}, (\rho_n)_{n \in \mathbb{N}}, (\kappa_{k}^{n})_{k \le n})$, we obtain a Thompson-like group, denoted as $\mathscr{T}_d(G_*)$. When $d=2$, we simply write $\mathscr{T}(G_*)$ for $\mathscr{T}_2(G_*)$. Sometimes we may say that the maps $((\rho_n)_{n \in \mathbb{N}}, (\kappa_{k}^{n})_{k \le n})$ ``form'' a $d$-ary cloning system on the sequence of groups $(G_n)_{n \in \mathbb{N}}$ or that the sequence of groups $(G_n)_{n \in \mathbb{N}}$ is ``equipped'' with the $d$-ary cloning system $((\rho_n)_{n \in \mathbb{N}}, (\kappa_k^n)_{k \le n})$ (as we already did in the introduction), which emphasizes the fact that a given sequence of groups may admit a variety of $d$-ary cloning systems, which we will see is certainly the case in Section \ref{fullycompatible}.

Before we recall the construction of the Thompson-like group $\mathscr{T}_d(G_*)$ arising from the $d$-ary cloning system $((G_n)_{n \in \mathbb{N}}, (\rho_n)_{n \in \mathbb{N}}, (\kappa_{k}^{n})_{k \le n})$, a few remarks are in order concerning the $d$-ary cloning system axioms. Although the $d$-ary cloning maps are not homomorphisms, axiom (C1) says that they are twisted homomorphisms with the twisting given by the representation maps $\rho_n$ (i.e., the action of each $G_n$ on $\{1,\dots,n\}$). As for axiom (C2), this axiom tells us that the $d$-ary cloning maps satisfy certain commutation relations which essentially mirror the relations in the infinite presentation of the Higman--Thompson group $F_d$: $$F_d = \langle x_0,x_1,x_2,\dots \mid x_{\ell} x_k = x_k x_{\ell + d-1}, k < \ell \rangle.$$ Finally, axiom (C3) essentially says that the $d$-ary cloning maps $(\kappa_k^n)_{k \le n}$ on $(G_n)_{n \in \mathbb{N}}$ have to be compatible with the standard $d$-ary cloning maps $(\varsigma_k^n)_{k \le n}$ on the symmetric groups $(S_n)_{n \in \mathbb{N}}$ in the sense that the two permutations $\rho_{n+d-1}((g)\kappa_k^n)$ and  $(\rho_n(g)) \varsigma_k^n$ have to agree almost everywhere on $\{1,\dots,n\}$ for every $g \in G_n$ and every $n \in \mathbb{N}$. It turns out that in many important and canonical examples (e.g.,  Higman--Thompson groups and their braided variants) these two permutations will actually agree everywhere (see Section \ref{fullycompatible}), but for the $d$-ary cloning system construction to produce a group with a well-defined binary operation it is not necessary that they agree everywhere. 

Let us now recall the construction of the group $\mathscr{T}_d(G_*)$. By a \emph{$d$-ary tree} we mean a finite rooted tree in which each non-leaf vertex has $d$ children, and a $d$-ary caret is a $d$-ary tree with $d$ leaves. A leaf is any vertex with no outgoing edges, which have to exist since we are considering finite trees. For the sake of convenience, we may sometimes say ``tree" instead of ``$d$-ary tree" when there is no risk of confusion. If $T$ is any $d$-ary tree, we let $n(T)$ denote the number of leaves of $T$. Given a $d$-ary tree $T$, let us call an expansion of $T$ the result of gluing the root of a $d$-ary caret to one of its $n(T)$ leaves, and then recursively extend the definition of extension to mean the result adding $d$-ary carets to its leaves an arbitrary but finite number of times. Recursively extending the notion of expansion allows us to essentially ``add" whole $d$-ary trees to any leaf of $T$ by gluing the root of the given $d$-ary tree to the desired leaf of $T$.

Let $((G_n)_{n \in \mathbb{N}}, (\rho_n)_{n \in \mathbb{N}}, (\kappa_k^n)_{k \le n})$ be a $d$-ary cloning system. Elements of $\mathscr{T}_d(G_*)$ are equivalence classes represented by a triple $(T,g,U)$, where $T$ and $U$ are $d$-ary trees with the same number of leaves, say $n \in \mathbb{N}$, and $g$ is an element of $G_n$. Let us now explain the equivalence relation which gives rise to such equivalence classes (i.e., elements of $\mathscr{T}_d(G_*)$). Given a triple $(T,g,U)$, an \emph{expansion} of it is a triple of the form $(T_{\rho_n(g)(k)},(g)\kappa_{k}^{n},U_k)$, where $U_k$ is obtained from $U$ by adding a $d$-ary caret to the $k$-th leaf and $T_{\rho_n(g)(k)}$ is obtained from $T$ by adding a $d$-ary caret to the $\rho_n(g)(k)$-th leaf. The equivalence relation on these triples is defined to be the symmetric and transitive closure of the expansions. Hence, roughly, two triples are equivalent if they differ by some finite sequence of expansions and reductions, where a reduction is simply a reverse expansion. We write $[T,g,U]$ for the equivalence class of $(T,g,U)$, which are the elements of $\mathscr{T}_d(G_*)$.

We have described what the elements of $\mathscr{T}_d(G_*)$ look like, but in order for this to form a group we must equip it with a binary operation. There is in fact a very natural binary operation on $\mathscr{T}_d(G_*)$ and here is how to define it. Given any two elements $[T,g,U]$ and $[V,h,W]$ in $\mathscr{T}_d(G_*)$, up to expansion we can assume $U = V$, because any pair (or even finite collection) of $d$-ary trees have a common expansion obtainable from either of them by adding the appropriate number of $d$-ary carets to the appropriate leaves (e.g., the  their ``union" is a common expansion). Now, multiplication of these two elements of $\mathscr{T}_d(G_*)$ is given by cancelling out the common tree and multiplying the group elements; i.e.,, $$[T,g,U][U,h,W] := [T,gh,W].$$ The $d$-ary cloning axioms guarantee that this is a well-defined group operation. The identity is $[T,1,T]$ for any $d$-ary tree $T$, and group inversion is given by $$[T,g,U]^{-1} = [U,g^{-1},T].$$

As we alluded to in the introduction, we said that the group $\mathscr{T}_d(G_*)$ can be regarded as a sort of Thompson-esque limit of the sequence of groups $(G_n)_{n \in \mathbb{N}}$, and as such it should contain copies of the groups $G_n$ in the limit, as is the case with most limit constructions. Indeed, given any $d$-ary tree $T$, we can define $$G_{T} := \{[T,g,T] : g \in G_{n(T)} \}$$ which is easily verified to be a subgroup of $\mathscr{T}_d(G_*)$. Then $g \mapsto [T,g,T]$ defines a group embedding of $G_{n(T)}$ into $\mathscr{T}_d(G_*)$ whose image is $G_{T}$.

If $((G_n)_{n \in \mathbb{N}}, (\rho_n)_{n \in \mathbb{N}}, (\kappa_k^n)_{k \le n})$ is any $d$-ary cloning system and $(H_n)_{n \in \mathbb{N}}$ is a sequence of subgroups such that the $d$-ary cloning system is ``compatible" with the sequence $(H_n)_{n \in \mathbb{N}}$, in the sense that $(H_n) \kappa_k^n \subseteq H_{n+d-1}$ for all $n \in \mathbb{N}$ and $1 \le k \le n$, then the $d$-ary cloning system on $(G_n)_{n \in \mathbb{N}}$ can be ``restricted" to form a $d$-ary cloning system on $(H_n)_{n \in \mathbb{N}}$. I.e., we obtain representation maps and $d$-ary cloning maps $\rho_n'$ and $(\kappa_k^n)'$, respectively, on $(H_n)_{n \in \mathbb{N}}$ by restricting $\rho_n$ and $\kappa_k^n$, respectively, and obtain a $d$-ary cloning system $((H_n)_{n \in \mathbb{N}}, (\rho_n')_{n \in \mathbb{N}}, ((\kappa_k^n)')_{k \le k})$ called a \textit{$d$-ary cloning subsystem} of $((G_n)_{n \in \mathbb{N}}, (\rho_n)_{n \in \mathbb{N}}, (\kappa_k^n)_{k \le n})$. The resulting Thompson-like group $\mathscr{T}_d(H_*)$ is a subgroup of $\mathscr{T}_d(G_*)$, which is ensured by the fact that the $H_n$ are stable under applying the $d$-ary cloning maps. 

Before we proceed any further, let us explain the standard $d$-ary cloning maps $\varsigma_k^n : S_{n} \to S_{n+d-1}$ on the sequence of finite symmetric groups referenced in the cloning axioms, and the Thompson-like group which results from them. Given $\sigma \in S_n$ a bijection of $\{1,\dots,n\}$, it can be represented as a diagram of arrows from one copy of $\{1,\dots,n\}$ up to a second copy with an arrow drawn from $i$ to $\sigma (i)$ for $i=1,\dots,n$. Then $(\sigma) \varsigma_k^n$, viewed in the same way as $\sigma$, is the diagram of arrows from $\{1,\dots,n, n+1,\dots,n+d-1\}$ up to a second copy of itself by taking the $k$-th arrow and replacing it by $d$ parallel arrows. See Figure~\ref{fig:symm_clone} for an example, and  for the rigorous, technical definition we refer to \cite[Example 2.2]{skipper21}.  Although somewhat tedious, it is not difficult to verify that $((S_n)_{n \in \mathbb{N}}, (\text{id}_{S_n})_{n \in \mathbb{N}}, (\varsigma_k^n)_{k \le n})$ forms a $d$-ary cloning system and that the resulting Thompson-like group is the Higman--Thompson group $V_d = \mathscr{T}_d(S_*)$. Since the subgroups $\langle (1~2~\cdots~n) \rangle \le S_n$ are stable under these $d$-ary cloning maps, we can restrict the standard $d$-ary cloning system on $(S_n)_{n \in \mathbb{N}}$ to $(\langle (1~2~\cdots~n) \rangle )_{n \in \mathbb{N}}$ with the resulting Thompson-like group being the Higman--Thompson group $T_d = \mathscr{T}_d(\langle (1~2~\cdots~*) \rangle)$. Finally when restricted to the trivial subgroup, the resulting Thompson-like group is the Higman--Thompson group $F_d = \mathscr{T}_d(\{1\})$.

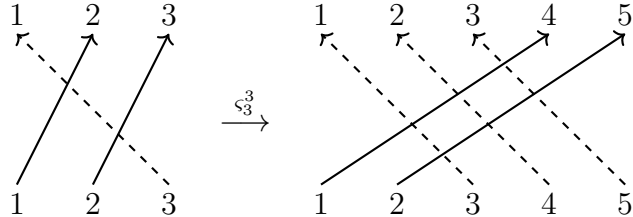
\begin{figure}[htb]
 \centering
 \begin{tikzpicture}[line width=0.8pt]
  \draw[->] (0,-2) -- (1,0); \draw[->] (1,-2) -- (2,0); \draw[->,dashed] (2,-2) -- (0,0);
  \node at (3,-1) {$\stackrel{\varsigma_3^3}{\longrightarrow}$};
  \node at (0,-2.25) {$1$};   \node at (0,0.25) {$1$};   \node at (1,-2.25) {$2$};   \node at (1,0.25) {$2$};   \node at (2,-2.25) {$3$};   \node at (2,0.25) {$3$};
  \begin{scope}[xshift=4cm]
   \draw[->] (0,-2) -- (3,0); \draw (1,-2)[->] -- (4,0); \draw (2,-2)[->,dashed] -- (0,0); \draw (3,-2)[->,dashed] -- (1,0); \draw (4,-2)[->,dashed] -- (2,0);
   \node at (0,-2.25) {$1$};   \node at (0,0.25) {$1$};   \node at (1,-2.25) {$2$};   \node at (1,0.25) {$2$};   \node at (2,-2.25) {$3$};   \node at (2,0.25) {$3$};   \node at (3,-2.25) {$4$};   \node at (3,0.25) {$4$};   \node at (4,-2.25) {$5$};   \node at (4,0.25) {$5$};
  \end{scope}
 \end{tikzpicture}
 \caption{An example of $3$-ary cloning in the symmetric groups. Here we see that $(1~2~3)\varsigma_3^3 = (1~4~2~5~3)$. The arrow getting ``cloned'' is dashed, as are the three arrows resulting from the cloning.}
 \label{fig:symm_clone}
\end{figure}

It is easy to deduce from the ``cloning a product" axiom that $(1)\kappa_k^n = 1$ for all $1 \le k \le n$ and all $n \in \mathbb{N}$, so any expansion of a triple of the form $(T,1,U)$ will be of the form $(T',1,U')$ for some $d$-ary trees $T'$ and $U'$. This entails that the set of all elements of the form $[T,1,U]$ defines a subgroup of $\mathscr{T}_d(G_*)$ for any $d$-ary cloning system $((G_n)_{n \in \mathbb{N}}, (\rho_n)_{n \in \mathbb{N}}, (\kappa_k^n)_{k \le n})$ and is in fact the canonical copy of $F_d$ inside the Thompson-like group $\mathscr{T}_d(G_*)$. Typically, we will write $[T,U]$ instead of $[T,1,U]$ out of convenience.

\subsection{Fully Compatible and (Slightly) Pure}\label{fullycompatible}

Although in the compatibility axiom (C2) the two permutations $\rho_{n+d-1}((g)\kappa_{k}^{n})$ and $(\rho_{n}(g)) \varsigma_{k}^{n}$ are not required to agree everywhere on $\{1,2,\dots,n\}$ for every $n \in \mathbb{N}$ and every $g \in G_n$, it turns out that in many natural examples (particularly, those which motivated the introduction of cloning systems) the two permutations do in fact agree everywhere, and when they do agree everywhere this leads to some nice properties. We note that there are some natural examples where these permutations do not equal each other for every $g \in G_n$ and every $n \in \mathbb{N}$. For example, the $d$-ary cloning systems giving rise to the R{\"o}ver--Nekrashevych groups are almost never fully compatible. However, since there are so many examples where the two permutations agree everywhere, let us encode this phenomenon into the following definition.

\begin{definition}\label{def:fullycompatible}
Call a $d$-ary cloning system $((G_n)_{n \in \mathbb{N}}, (\rho_n)_{n \in \mathbb{N}}, (\kappa_k^n)_{k \le n})$ \emph{fully compatible} if $\rho_{n+d-1}((g)\kappa_k^n)(i)=(\rho_n(g))\varsigma_k^n(i)$ for all $1 \le k \le n$ and all $1 \le i \le n+d-1$ even when $i=k,k+1,\dots,k+d-1$, or, more simply, we have the commutation relation $\varsigma_k^n \circ 
 \rho_n = \rho_{n+d-1} \circ \kappa_k^n$ of the representation maps (the action) and the $d$-ary cloning maps for all $1 \le k \le n$ and $n \in \mathbb{N}$.
\end{definition}

Trivially, the standard $d$-ary cloning system on $(S_n)_{n \in \mathbb{N}}$ giving rise to $V_d$ is fully compatible. For another example, the standard $d$-ary cloning system on the braid groups $(B_n)_{n \in \mathbb{N}}$ is fully compatible, and this $d$-ary cloning system gives rise to the braided Higman--Thompson group $bV_d := \mathscr{T}_d(B_*)$. Let us explain the standard $d$-ary cloning system on the braid groups. First, given a braid $b \in B_n$, we can number the bottom and top of the strands from left to right with $\{1,2,\dots,n\}$. The representation maps $B_n \to S_n$ are the standard projections tracking the number where a given strand of $b$ starts and ends. As for the $d$-ary cloning maps on $(B_n)_{n \in \mathbb{N}}$, $\kappa_k^n : B_n \to B_{n+d-1}$ sends $b \in B_n$ to the braid $(b) \kappa_k^n$ of $B_{n+d-1}$ obtained by replacing the $k$-th strand with $d$ parallel strands having the same crossing type. As we stated, this forms a fully compatible $d$-ary cloning system on braid groups $(B_n)_{n \in \mathbb{N}}$. See Figure~\ref{fig:braid} for an example of an expansion of a triple representing an element in $bV = bV_2$.

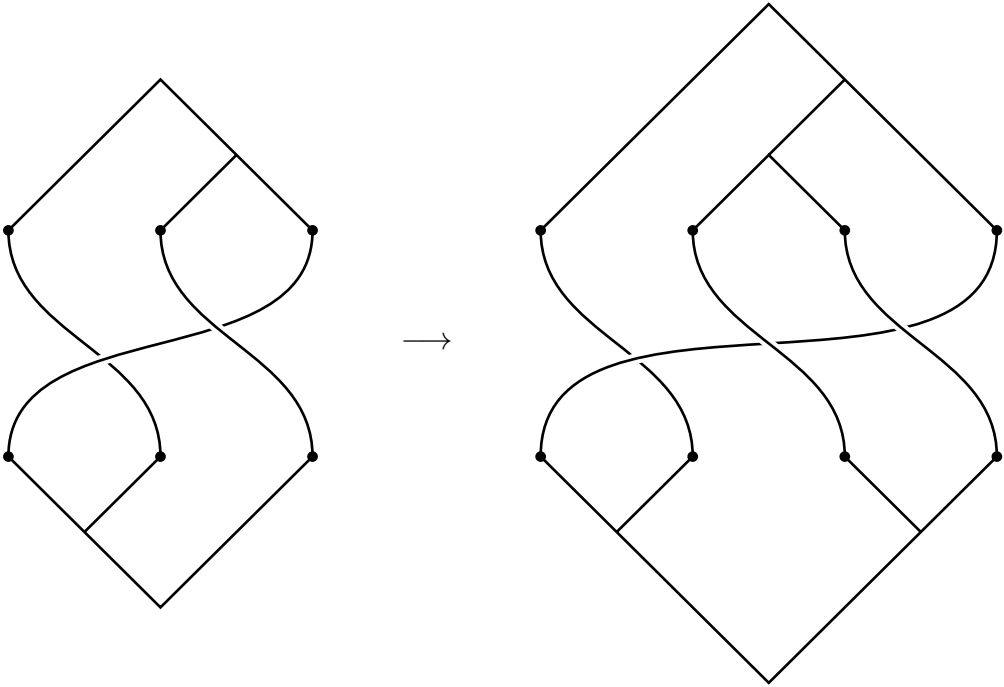
\begin{figure}[htb]
 \begin{tikzpicture}[line width=1pt]\centering
  \draw (0,0) -- (2,2) -- (4,0)   (3,1) -- (2,0);
	
	\draw[white, line width=4pt]
  (0,0) to [out=-90, in=90] (2,-3);
  \draw
  (0,0) to [out=-90, in=90] (2,-3);
	
	\draw[white, line width=4pt]
  (4,0) to [out=-90, in=90] (0,-3);
  \draw
  (4,0) to [out=-90, in=90] (0,-3);
	
  \draw[white, line width=4pt]
  (2,0) to [out=-90, in=90] (4,-3);
  \draw
  (2,0) to [out=-90, in=90] (4,-3);
	
	\filldraw (0,0) circle (1.5pt);
	\filldraw (2,0) circle (1.5pt);
	\filldraw (4,0) circle (1.5pt);
	
	\draw (0,-3) -- (2,-5) -- (4,-3)   (1,-4) -- (2,-3);
	\filldraw (0,-3) circle (1.5pt);
	\filldraw (2,-3) circle (1.5pt);
	\filldraw (4,-3) circle (1.5pt);
	
	\node at (5.5,-1.5) {$\longrightarrow$};
	
	\begin{scope}[xshift=8cm,yshift=1cm]	
   \draw (-1,-1) -- (2,2) -- (5,-1)   (3,1) -- (2,0)   (1,-1) -- (2,0) -- (3,-1);
	 
   \draw[white, line width=4pt]
   (-1,-1) to [out=-90, in=90] (1,-4);
   \draw
   (-1,-1) to [out=-90, in=90] (1,-4);
	
   \draw[white, line width=4pt]
   (5,-1) to [out=-90, in=90] (-1,-4);
   \draw
   (5,-1) to [out=-90, in=90] (-1,-4);
	
   \draw[white, line width=4pt]
   (1,-1) to [out=-90, in=90] (3,-4);
   \draw
   (1,-1) to [out=-90, in=90] (3,-4);
	
   \draw[white, line width=4pt]
   (3,-1) to [out=-90, in=90] (5,-4);
   \draw
   (3,-1) to [out=-90, in=90] (5,-4);
	
	 \filldraw (-1,-1) circle (1.5pt);
	 \filldraw (1,-1) circle (1.5pt);
	 \filldraw (3,-1) circle (1.5pt);
	 \filldraw (5,-1) circle (1.5pt);
	
	 \draw (-1,-4) -- (2,-7) -- (5,-4)   (0,-5) -- (1,-4)   (3,-4) -- (4,-5);
	 \filldraw (-1,-4) circle (1.5pt);
	 \filldraw (1,-4) circle (1.5pt);
	 \filldraw (3,-4) circle (1.5pt);
	 \filldraw (5,-4) circle (1.5pt);
	\end{scope}
 \end{tikzpicture}
 \caption{An example of an expansion using the standard $2$-ary cloning system defined on the braid groups $(B_n)_{n \in \mathbb{N}}$. We represent $(T_-,b,T_+)$ by drawing $T_+$ upside-down and below $T_-$, connecting the leaves of $T_-$ and $T_+$ with the strands of braid the $b$. Since these pictures differ by an expansion, they represent the same element of $bV=\mathscr{T}(B_*)$.}\label{fig:braid}
\end{figure}

Let us now explore some basic properties of fully compatible $d$-ary cloning systems. For such $d$-ary cloning systems, it is not too difficult to see that if $g \in \ker (\rho_n)$, then $(g) \kappa_k^n \in \ker (\rho_{n+d-1})$ for any $1 \le k \le n$ and $n \in \mathbb{N}$. In particular, given a triple of the form $(T,g,T)$ for some $d$-ary tree $T$ and $g \in \ker (\rho_{n(T)})$, it follows from this observation that any expansion of $(T,g,T)$ must be of the form $(T_k, (g)\kappa_k^{n(T)},T_k)$ for some $1 \le k \le n(T)$, which is clearly of the same form as $(T,g,T)$. Note, $T_k$ denotes the $d$-ary tree obtained from $T$ by adding a $d$-ary caret to $k$-th leaf of $T$. This shows that $$\mathscr{K}_d(G_*) := \{[T,g,T] : g \in \ker (\rho_{n(T)}) \}$$ is a subgroup of $\mathscr{T}_d(G_*)$, and in fact it is a normal subgroup of $\mathscr{T}_d(G_*)$. If $T$ is any $d$-ary tree and $K_T$ is defined to be the kernel of the map from $G_T \to S_{n(T)}$ given by $[T,g,T] \mapsto \rho_{n(T)}(g)$, which is well-defined since we are assuming the $d$-ary cloning system is fully compatible, then it is clear that $\mathscr{K}_d(G_*)$ is a directed union of the subgroups $K_T$ as $T$ varies over all $d$-ary trees. For fully compatible $d$-ary cloning systems, we have a natural group homomorphism $\pi : \mathscr{T}_d(G_*) \to V_d$ and in fact $\mathscr{K}_d(G_*)$ is the kernel of the homomorphism. We have the following lemma to record these facts about $\pi$ and $\mathscr{K}_d(G_*)$, the proof of which can be found in \cite[Lemma 3.2]{witzel18} (at least, the proof of the $d=2$ case which easily generalizes to the $d > 2$ case).

\begin{lemma}[Map to $V_d$]
Given a fully compatible $d$-ary cloning system \newline $((G_n)_{n \in \mathbb{N}}, (\rho_n)_{n \in \mathbb{N}}, (\kappa_k^n)_{k \le n})$, there is a homomorphism $\pi : \mathscr{T}_d(G_*) \to V_d$ given by 

$$[T, g,U] \mapsto [T,\rho_{n(T)}(g), U],$$ 
where the kernel is $\mathscr{K}_d(G_*)$ and the image is some group $W_d$ with $F_d \le W_d \le V_d$. As a matter of fact, we have a short exact sequence $$1 \xrightarrow{} \mathscr{K}_d(G_*) \xrightarrow{} \mathscr{T}_d(G_*) \xrightarrow{\pi} W_d \xrightarrow{} 1.$$
\end{lemma}

The fact that we have a short exact sequence means that when we form the group von Neumann algebra of $\mathscr{T}_d(G_*)$, we get the twisted crossed product decomposition of $L(\mathscr{T}_d(G_*))$ as $$L(\mathscr{T}_d(G_*)) \cong L(\mathscr{K}_d(G_*)) \rtimes_{(\sigma ,v)} W_d$$ 
for some coycle action $(\sigma,v)$. Although we will not use this fact in this paper,  it is worth recording nonetheless. 

A very important subclass of fully compatible $d$-ary cloning systems is the class of so-called pure $d$-ary cloning systems:

\begin{definition}\label{def:pure}
Call a $d$-ary cloning system $((G_n)_{n \in \mathbb{N}}, (\rho_n)_{n \in \mathbb{N}}, (\kappa_k^n)_{k \le n})$ \emph{pure} provided all the representation maps $\rho_n$ are trivial; i.e., $\rho_n(g)$ is the identity permutation on $\{1,\dots,n\}$ for every $g \in G_n$ and every $n \in \mathbb{N}$. 
\end{definition}

For pure $d$-ary cloning systems, it is clear that the cloning maps comprising the $d$-ary cloning system must actually be injective group homomorphisms (this follows immediately from axiom (C1)), the image of $\pi$ is $F_d$, and the short exact sequence splits, so we get the following internal semi-direct product decomposition: $$\mathscr{T}_d(G_*) = \mathscr{K}_d(G_*) \rtimes F_d.$$

Informally, we will regard Thompson-like groups arising from pure $d$-ary cloning systems as ``$F$-like", as they are split extensions of $F_d$. On the other hand, those arising from $d$-ary cloning systems in which the representation maps $\rho_n : G_n \to S_n$ are surjective for every $n \in \mathbb{N}$ are regarded as ``$V$-like". Indeed, in the case when the representation maps $\rho_n$ are surjective and the $d$-ary cloning system is fully compatible, it is easy to see that the image of $\pi$ is $V_d$, meaning $\mathscr{T}_d(G_*)$ is a group extension of $V_d$. 

If $((G_n)_{n \in \mathbb{N}}, (\rho_n)_{n \in \mathbb{N}}, (\kappa_k^n)_{k \le n})$ is any fully compatible $d$-ary cloning system, then the fully compatible assumption insures that $(\ker \rho_n) \kappa_k^n \subseteq \ker \rho_{n+d-1}$ as we already noted above. Since the kernels are stable under the $d$-ary cloning maps, this entails that the $d$-ary cloning system on $(G_n)_{n \in \mathbb{N}}$ can be restricted to the sequence of subgroups $(\ker \rho_n)_{n \in \mathbb{N}}$ and clearly this $d$-ary cloning system is pure. Hence, given any fully compatible $d$-ary cloning system, we can form a pure $d$-ary cloning system $((\ker \rho_n)_{n \in \mathbb{N}}, (\rho_n)_{n \in \mathbb{N}}, (\kappa_k^n)_{k \le n})$ which one might call the ``maximal pure cloning subsystem" associated to the fully compatible $d$-ary cloning system $((G_n)_{n \in \mathbb{N}}, (\rho_n)_{n \in \mathbb{N}}, (\kappa_k^n)_{k \le n})$.  For example, the kernels of the representation maps of the standard $d$-ary cloning system on the braid groups $(B_n)_{n \in \mathbb{N}}$ are the pure braid groups $(PB_n)_{n \in \mathbb{N}}$, and the resulting group formed by restricting this $d$-ary cloning system to $(PB_n)_{n \in \mathbb{N}}$ is the pure braided Higman--Thompson group $bF_d := \mathscr{T}_d(PB_*)$. Since the representation maps in the $d$-ary cloning systems on the braid groups are surjective, $bV_d$ is regarded as $V$-like, whereas $bF_d$ is regarded as $F$-like. 

The following example of $d$-ary cloning systems is an extremely important family of pure $d$-ary cloning systems, and we will constantly use and refer to them throughout the paper as an interesting source of examples and non-examples. 

\begin{example}[Direct products with monomorphisms]
Let $G$ be any group and $$\phi_1,\dots,\phi_d : G \to G$$ be a family of monomorphisms. Let $\prod^n(G)$ be the direct product of $n$ copies of $G$, and let $\rho_n : \prod^n(G) \to S_n$ be trivial. For $n \in \mathbb{N}$ and $1 \le k \le n$, define the $k$-th $d$-ary cloning map $\kappa_k^n : \prod^n(G) \to \prod^{n+d-1}(G)$ via
$$(g_1,\dots,g_n) \kappa_k^n = (g_1,\dots,g_{k-1}, \phi_1(g_k),\dots,\phi_d(g_k),g_{k+1},\dots,g_n)$$
It is relatively straightfoward to verify that this forms a $d$-ary cloning system on the sequence of groups $(\prod^n(G))_{n \in \mathbb{N}}$ which is by construction pure in the above sense. We note that the $d=2$ case was first considered by Tanushevski in \cite{tanushevski16}.
\end{example}

Finally, there is a generalization of pure $d$-ary cloning systems that does not completely fall under the umbrella of fully compatible $d$-ary cloning systems:

\begin{definition}[Slightly Pure]
Call a $d$-ary cloning system $((G_n)_{n \in \mathbb{N}}, (\rho_n)_{n \in \mathbb{N}}, (\kappa_k^n)_{k \le n})$ \emph{slightly pure} if for all $n \in \mathbb{N}$ and $g \in G_n$, we have $\rho_n(g)(n)=n$. 
\end{definition}

If $\rho_n(g)$ is the identity permutation, then clearly, in particular, it fixes $n$. Hence, all pure $d$-ary cloning systems are slightly pure. For a non-pure example, consider the subgroup $\widehat{S}_n \le S_n$ of permutations fixing $n$. The standard $d$-ary cloning system on $(S_n)_{n \in \mathbb{N}}$ restricts to $(\widehat{S}_n)_{n \in \mathbb{N}}$, and this restricted $d$-ary cloning system is slightly pure. The resulting Thompson-like group is denoted by $\widehat{V}_d := \mathscr{T}_d(\widehat{S}_*)$. Although this may seem like a rather ad hoc example at first glance, this group does appear in the literature (at least, e.g., $\widehat{V} = \widehat{V}_2$ was first considered in \cite{brin07}). Finally, for non-fully compatible examples, the R{\"o}ver--Nekrashevych groups arise from $d$-ary cloning systems which are almost never fully compatible but can be slightly pure (see \cite[Lemma 5.3]{skipper19}).

\subsection{Diverse and Uniform $d$-ary Cloning Systems}\label{diverse}

The next two properties we consider concern the $d$-ary cloning maps, while the  properties in the last section concerned the representation maps. These two properties were crucial in proving Citation~\ref{factor} and Citation~\ref{mcduff}, especially the diversity assumption. Throughout the course of this paper, we will see quite plainly how powerful, albeit modest, the diversity assumption is. As a matter of fact, the diversity assumption alone will be enough to prove that the inclusion $L(F_d) \subseteq L(\mathscr{T}_d(G_*))$ is irreducible (and hence $L(\mathscr{T}_d(G_*))$ is a type $\II_1$ factor), is singular, and even stronger it satisfies the WAHP. Many $d$-ary cloning systems have the property of being diverse, a property which we now define. 

\begin{definition}[Diverse]
A $d$-ary cloning system $((G_n)_{n \in \mathbb{N}}, (\rho_n)_{n \in \mathbb{N}}, (\kappa_k^n)_{k \le n})$ is said to be \emph{diverse} provided there exists a natural number $n_0 \in \mathbb{N}$ such that $$\bigcap_{k=1}^{n} \image \kappa_k^n = \{1\}$$ for all $n \ge n_0$.
\end{definition}

Intuitively, the diversity assumption says that the different cloning maps $G_n \to G_{n+d-1}$ tend to send $G_n$ into different regions of $G_{n+d-1}$. The diversity assumption is a very natural assumption and as might be expected many $d$-ary cloning systems satisfy this property. For example, the standard $d$-ary cloning system on $(S_n)_{n \in \mathbb{N}}$ is diverse, the standard $d$-ary cloning system on the braid groups $(B_n)_{n \in \mathbb{N}}$ is diverse (see \cite[Proposition 4.2]{bashwinger}), and some of the R{\"o}ver--Nekrashevych groups arise from diverse $d$-ary cloning systems.

We can also produce many interesting diverse $d$-ary cloning systems using the direct products with injective monomorphisms. For example, if the images of the monomorphisms $\phi_1,\dots,\phi_d : G \to G$ have trivial intersection (meaning, the intersection is $\{1\}$), then the $d$-ary cloning system is diverse. See \cite[Example 4.1]{bashwinger} for concrete examples of groups admitting monomorphisms which intersect trivially. 

There is, however, a way to bypass the restriction that the images of the monomorphisms $\phi_1,\dots,\phi_d : G \to G$ intersect trivially if we look to a certain subgroup of $\prod^n(G)$. Indeed, define $\Psi^n(G) := \{1 \} \times \prod^{n-1}(G) \le \prod^n(G)$. If $\phi_1,\dots,\phi_d : G \to G$ are any monomorphisms with the $d$-ary cloning system on $(\prod^n(G))_{n \in \mathbb{N}}$ restricted to $(\Psi^n(G))_{n \in \mathbb{N}}$, then this restricted $d$-ary cloning system on $(\Psi^n(G))_{n \in \mathbb{N}}$ is diverse. To see this, suppose that $$(1,g_2\dots,g_{n+d-1}) \in \bigcap_{k=1}^{n} \image \kappa_k^n$$
This means, in particular, that $(1,g_2,\dots,g_{n+d-1}) \in \image \kappa_1^n$ which implies $g_2 = \dots = g_{d} = 1$. Then, using the fact that $(1,\dots,1,g_{d+2},\dots,g_{n+d-1}) \in \image \kappa_2^n$, we can argue that $g_{d+1}=1$. Continuing in this fashion, we can argue that the rest of the $g_i$'s are the identity. Hence, the $d$-ary cloning system on $(\Psi^n(G))_{n \in \mathbb{N}}$ is diverse, and according to Theorems~\ref{thrm:Irreducible} and~\ref{thrm:WAHP}, the inclusion $L(F_d) \subseteq L(\mathscr{T}_d(\Psi^*(G)))$ will satisfy the WAHP and hence be singular and irreducible for any choice of group $G$ and any choice of monomorphisms $\phi_1,\dots,\phi_d : G \to G$. 

This is not the only remarkable fact about these examples. As we noted in \cite[Example 5.12]{bashwinger}, if you further assume that every monomorphism $\phi_1,\dots,\phi_d$ is the identity morphism, then the $d$-ary cloning system on $(\Psi^n(G))_{n \in \mathbb{N}}$ is also uniform (see Definition~\ref{uniform} below) and hence Citation~\ref{mcduff} tells us that $L(\mathscr{T}_d(\Psi^*(G)))$ is a type $\II_1$ McDuff factor and therefore $\mathscr{T}_d(\Psi^*(G))$ is inner amenable for any choice of group $G$. It would be interesting to find a group $G$ and monomorphisms $\phi_1,\dots,\phi_d$ such that the $d$-ary cloning system on $(\Psi^n(G))_{n \in \mathbb{N}}$ is not uniform but yet $L(\mathscr{T}_d(\Psi^*(G)))$ is a type $\II_1$ McDuff factor, has property Gamma, or at least $\mathscr{T}_d(\Psi^*(G))$ is inner amenable. Perhaps we can construct a finitely generated ``Vaes" group, i.e., an inner amenable ICC group whose group von Neumann algebra does not have property Gamma, which was predicted to exist by Vaes in \cite{vaes12}.

In summary, Table~\ref{tab:inclusions} shows many examples where inclusions of the form $L(F_d) \subseteq L(\mathscr{T}_d(G_*))$ are irreducible, singular, and even satisfy the WAHP by virtue of the corresponding $d$-ary cloning system giving rise to $\mathscr{T}_d(G_*)$ being diverse. Note, that we do not define the groups $\mathscr{T}_d(\overline{B}_*(R))$ and $\mathscr{T}_d(\text{Ab}_*)$ referenced in the last two rows of table for the purpose of saving space; we refer to \cite{bashwinger} for the explanation of these groups and the $d$-ary cloning systems giving rise to them.

\begin{table}
\centering
\begin{tabular}{|||c || p{12cm}|||}
\hline & \\
\hline
$L(F_d) \subseteq L(T_d)$   & $T_d$ Higman--Thompson group \\
$L(F_d) \subseteq L(V_d)$     & $V_d$ Higman--Thompson group \\
$L(F_d) \subseteq L(bF_d)$   & $bF_d$ Braided-Higman Thompson group \\
$L(F_d) \subseteq L(bV_d)$   & $bV_d$ Braided-Higman--Thompson group \\
$L(F_d) \subseteq L(\mathscr{T}_d(\prod^*(G)))$ & $G$ any group any $\phi_i : G \to G$ monomorphisms whose images intersect trivially \\
$L(F_d) \subseteq L(\mathscr{T}_d(\Psi^*(G_*)))$ & $G$ any group and $\phi_i : G \to G$ any monomorphisms \\
$L(F_d) \subseteq L(\widehat{V}_d)$ & $\widehat{V}_d$ is a certain variation on $V_d$\\ 
$L(F_d) \subseteq L(\mathscr{T}_d(\overline{B}_*(R)))$ & $\overline{B}_n(R)$ is the group of $n \times n$ upper-triangular matrices over any countable ring $R$ modulo the center \\
$L(F_d) \subseteq L(\mathscr{T}_d(\text{Ab}_*))$ & $\text{Ab}_n$ is the $n$-th Abels group \\
\hline & \\
\hline
\end{tabular}   
\caption{A list of a inclusions satisfying the weak asymptotic homomorphism property as a consequence of the corresponding $d$-ary cloning system being diverse, and hence inclusions which are singular and irreducible.}
\label{tab:inclusions}
\end{table}

The other property that was crucial in proving that $\mathscr{T}_d(G_*)$ can yield a type $\II_1$ McDuff factor is the uniformity property. In Section \ref{mixing}, we will see that there is a connection between the uniformity property and mixing of the inclusion $L(F_d) \subseteq L(\mathscr{T}_d(G_*))$, specifically that the inclusion never exhibits mixing whenever $(G_n)_{n \in \mathbb{N}}$ is a equipped with a uniform $d$-ary cloning system. Let us now define the uniformity property.

\begin{definition}[Uniform]\label{uniform}
A $d$-ary cloning system $((G_n)_{n \in \mathbb{N}}, (\rho_n)_{n \in \mathbb{N}}, (\kappa_k^n)_{k \le n})$ is said to be \emph{uniform} provided that for all $1 \le k \le n$ and $\ell, \ell'$ satisftying $k \le \ell \le \ell ' \le k+d-1$, we have $$\kappa_k^n \circ \kappa_\ell^{n+d-1} = \kappa_k^n \circ \kappa_{\ell '}^{n+d-1}.$$ 
\end{definition}

Intuitively, the uniformity property says if we clone an element and then clone a part of it that was involved in the first cloning, it doesn’t matter which part of it we use. Hence, as we noted above, it is clear that if all the monomorphisms $\phi_1,\dots,\phi_d$ are the identity morphism, then the $d$-ary cloning system on the direct products will be uniform. For another example, the standard $d$-ary cloning system on the braid groups is uniform. For example, in the standard 2-ary cloning system on the braid groups, if we clone the 3rd strand to create two parallel strands, and then follow that up by cloning one of the new
strands, either the 3rd or 4th, it doesn’t matter which one we clone. Either way we will end up effectively having turned the original 3rd strand into three new parallel strands. 

Finally, we note that if a $d$-ary cloning has any of the above mentioned properties, such as being fully compatible, (slightly) pure, diverse, or uniform, any $d$-ary cloning subsystem will inherit the respective properties. This can be easily verified by inspecting the definitions. 

\section{On the structure of the inclusion $L(F_d) \subseteq L(\mathscr{T}_d(G_*))$}\label{structure}

In this section, we prove some results about the most natural subfactor of $L(\mathscr{T}_d(G_*))$ one can single out---namely, the Higman--Thompson group factor $L(F_d)$. More specifically, we analyze the inclusion
$$L(F_d) \subseteq L(\mathscr{T}_d(G_*)).$$ 
What we show is that if $(G_n)_{n \in \mathbb{N}}$ is sequence of groups equipped with a diverse $d$-ary cloning system, then the above inclusion satisfies the WAHP. Then we study irreducibility, singularity, and the WAHP of the above inclusion with respect to the R{\"o}ver--Nekrashevych groups. Then we investigate the relationship, or rather the lack of any relationship, between the diversity assumption and mixing of the inclusion $L(F_d) \subseteq L(\mathscr{T}_d(G_*))$. Finally, we investigate the above inclusion for some non-diverse $d$-ary cloning systems with respect to the WAHP.

\subsection{$L(F_d) \subseteq L(\mathscr{T}_d(G_*))$ satisfies the Weak Asymptotic Homomorphism Property}\label{WAHP}

The main result of this section is theorem ~\ref{thrm:WAHP} which states if $(G_n)_{n \in \mathbb{N}}$ is a sequence of groups equipped with a diverse $d$-ary cloning system, then the inclusion $L(F_d) \subseteq L(\mathscr{T}_d(G_*))$ satisfies the WAHP. As we alluded to in Section \ref{normalizers}, Jolissaint noted in \cite{Jolissaint12a} that an inclusion of tracial von Neumann algebras satisfying the WAHP is precisely the same thing as the inclusion being weakly mixing. Thus, our result can be rephrased as saying that $L(F_d)$ is a weakly mixing subfactor of $L(\mathscr{T}_d(G_*))$ whenever $(G_n)_{n \in \mathbb{N}}$ is equipped with a diverse $d$-ary cloning system.

Before we can prove prove Theorem~\ref{thrm:WAHP}, we need to prove a few lemmas. Our first lemma is about forming powers of certain $d$-ary tree pairs (i.e., elements of $F_d$). In order to prove this lemma, we first need to deduce a ``commutation" relation on expansions of a single $d$-ary tree. Given a $d$-ary tree $T$, recall that $T_\ell$ denotes the tree obtained by adding a $d$-ary caret to the $\ell$-th leaf of $T$, where $1 \le \ell \le n(T)$. Let $(T_\ell)_{k}$ denote the tree obtained from $T_\ell$ by adding a $d$-ary to the $k$-th leaf, where $1 \le k \le n(T)+d-1$. This notation can be obviously extended to any finite number of expansions using nested parenthesis, and we will use this notation quite extensively. Now, if $1 \le k < \ell \le n(T)$, then it is clear that $(T_\ell)_k = (T_k)_{\ell+d-1}$. This is because to get from $T$ to $T_k$ we add a caret to the $k$-th leaf of $T$. Now, the leaf that used to be the $\ell$-th leaf becomes the $(\ell+d-1)$-th leaf. Thus, $(T_k)_{\ell+d-1}$ can also be obtained from $T$ by first adding a caret to the $\ell$-th leaf and then adding a caret to the $k$-th leaf, which is to say that $(T_\ell)_k = (T_k)_{\ell + d - 1}$. Notice the commutation relation $(T_\ell )_k = (T_k)_{\ell + d-1}$ essentially mirrors the commutation relation in the infinite presentation of $F_d$, which ought to come as no surprise. This observation will be helpful in the proof of the following lemma. 

\begin{lemma}\label{powers}
Let $T$ be a $d$-ary tree with $n$ leaves, and let $1 \le k < \ell \le n$. Then 

$$[T_k, T_\ell]^{m} = [\underbrace{(\dots(T_{k})_{k}\dots)_{k}}_{(m-1) \text{ times}}, (\dots((T_\ell)_{\ell+d-1})_{\ell + 2(d-1)}\dots)_{\ell+(m-1)(d-1)}]$$
for every $m \in \mathbb{N}$.
\end{lemma}
\begin{proof}
The above observation allows us to easily handle the base case $m=2$. Indeed,

\begin{align*}
[T_k,T_\ell]^2 & = [T_k,T_\ell][T_k, T_\ell] \\
&= [(T_k)_k, (T_\ell)_k] [(T_k)_{\ell + d-1}, (T_\ell)_{\ell + d-1}] \\
&= [(T_k)_k, (T_\ell)_k] [(T_\ell)_k, (T_\ell)_{\ell + d-1}] \\
&=[(T_k)_k, (T_\ell )_{\ell + d-1}].
\end{align*}
Now let us deal with the inductive case:

\begin{align*}
[T_k, T_\ell]^{m+1} &= [T_k, T_\ell]^{m} [T_k,T_\ell] \\
&= [\underbrace{(\dots(T_{k})_{k}\dots)_{k}}_{(m-1) \text{ times}}, (\dots((T_\ell)_{\ell+d-1})_{\ell + 2(d-1)}\dots)_{\ell+(m-1)(d-1)}] [T_k, T_\ell].
\end{align*}
Expand the tree pair on the left in the product by adding a $d$-ary caret at the $k$-th leaf: \begin{align*}
[T_k, T_\ell]^{m+1} &= [\underbrace{(\dots(T_{k})_{k}\dots)_{k}}_{m \text{ times}}, ((\dots((T_\ell)_{\ell+d-1})_{\ell + 2(d-1)}\dots)_{\ell+(m-1)(d-1)})_k] [T_k, T_\ell].
\end{align*}
For the leftmost factor, we use the above observation to ``propagate" the index $k$  towards $T$, making sure that we add $d-1$ to the index at each step in propagating the $k$ towards $T$. As for the rightmost factor, we expand the $d$-ary tree pair at the $(\ell + d-1)$-th spot, then the $(\ell + 2 (d-1))$-th spot, then the $(\ell + 3(d-1))$-th spot, and so on. Doing this we obtain

\begin{align*}
[T_k, T_\ell]^{m+1} &= [\underbrace{(\dots(T_{k})_{k}\dots)_{k}}_{m \text{ times}}, (\dots(T_k)_{\ell + d-1} \dots )_{\ell + m(d-1)}]  \\
& [(\dots(T_k)_{\ell + d-1} \dots )_{\ell + m(d-1)}, (\dots(T_\ell)_{\ell + d-1}\dots)_{\ell + m(d-1)}] \\
&= [\underbrace{(\dots(T_{k})_{k}\dots)_{k}}_{m \text{ times}},(\dots(T_\ell)_{\ell + d-1}\dots)_{\ell + m(d-1)}].
\end{align*}
Note that the product had to be written on two lines. This finishes the inductive case and hence the lemma is proved.

\end{proof}

Now let us investigate how group inversion and $d$-ary cloning maps (and iterated compositions of them) commute. First, from the cloning a product axiom (axiom (C2)), we immediately get $$1 = (g^{-1}g) \kappa_k^n = (g^{-1})\kappa_{\rho_n(g)k} \cdot (g) \kappa_k^n$$
or $$\big((g)\kappa_k^n\big)^{-1} = (g^{-1})\kappa_{\rho_n(g)k}^n.$$

Now let $k_1 \in \{1,\dots,n\}$ and $k_2 \in \{ 1,\dots,n,n+1,\dots,n+d-1\}$. Then 

\begin{align*}
\bigg[(g)(\kappa_{k_1}^n \circ \kappa_{k_2}^{n+d-1}) \bigg]^{-1} &=  \bigg[\bigg((g)\kappa_{k_1}^n \bigg)\kappa_{k_2}^{n+d-1} \bigg]^{-1} \\
&= \bigg( \big((g)\kappa_{k_1}^n\big)^{-1} \bigg) \kappa_{\rho_{n+d-1}((g)\kappa_{k_1}^n)k_2}^{n+d-1} \\
&= \bigg( (g^{-1})\kappa_{\rho_n(g)k_1}^{n} \bigg) \kappa_{\rho_{n+d-1}((g)\kappa_{k_1}^n)k_2}^{n+d-1} \\
&= (g^{-1})\bigg( \kappa_{\rho_n(g)k_1}^{n} \circ \kappa_{\rho_{n+d-1}((g)\kappa_{k_1}^n)k_2}^{n+d-1}\bigg)\\
\end{align*}
Hence, arguing inductively we have the following lemma: 

\begin{lemma}\label{cloning_inverse}
Let $((G_n)_{n \in \mathbb{N}}, (\rho_n)_{n \in \mathbb{N}}, (\kappa_k^n)_{k \le n})$ be a $d$-ary cloning system, and let $n \in \mathbb{N}$ and $g \in G_n$ be arbitrary. For any natural number $m \ge 2$ and a sequence of natural numbers $k_1,k_2,\dots,k_m$ with $k_i \in \{1,\dots,n+(i-1)(d-1)\}$ we have $$\bigg[(g)\big(\kappa_{k_1}^n \circ \kappa_{k_2}^{n+d-1} \circ \dots \circ \kappa_{k_m}^{n+(m-1)(d-1)}\big) \bigg]^{-1} = (g^{-1})\big(\kappa_{\alpha_1}^n \circ \kappa_{\alpha_2}^{n+d-1} \circ \dots \circ \kappa_{\alpha_m}^{n+(m-1)(d-1)}\big),$$
where $\alpha_1 = \rho_n(g)k_1$ and $$\alpha_i = \rho_{n+(i-1)(d-1)} \bigg( (g)\big(\kappa_{k_1}^{n} \circ \kappa_{k_2}^{n+d-1} \circ \dots \circ \kappa_{k_{i-1}}^{n+(i-2)(d-1)}\big) \bigg)k_i$$
for $i=2,\dots,m$. 
\end{lemma}

This lemma tells us that the group inverse ``propagates" towards the group element $g$ with the appropriate twisting by $g$ under the appropriate representation maps and (sequence of) $d$-ary cloning maps at each stage. Because the $\rho_n(g)$ are bijections, if we are given $\alpha_1,\dots,\alpha_m$, we can recursively choose the $k_i$ to satisfy the equations. This will be important for the proof of Theorem~\ref{thrm:Irreducible} because we are going to expand a certain triple $(T,g,U)$ in multiple ways, but we want to add $d$-ary carets to $T$ in a prescribed manner which will require us to carefully add $d$-ary carets to $U$ in a specific way, which is possible thanks to $\rho_n(g)$ being a bijection. 

Finally, we can prove the main result of this section.

\begin{theorem}\label{thrm:WAHP}
Let $((G_n)_{n \in \mathbb{N}}, (\rho_n)_{n \in \mathbb{N}}, (\kappa_k^n)_{k \le n})$ be a diverse $d$-ary cloning system. Then the inclusion $L(F_d) \subseteq L(\mathscr{T}_d(G_*))$ satisfies the weak asymptotic homomorphism property. 
\end{theorem}
\begin{proof}
Let $x, x_1,\dots,x_t \in \mathscr{T}_d(G_*)$ be such that $$F_dx \subseteq x_1 F_d \cup \dots \cup x_t F_d.$$
Since the right hand size is invariant under multiplication on the right by elements from $F_d$, the above implies $$F_dxF_d \subseteq x_1 F_d \cup \dots \cup x_t F_d.$$
If we consider the left-translation action of $F_d$ on the left coset space $\mathscr{T}_d(G_*)/F_d$, the above says that the orbit of $xF_d$ under this action is finite. By the orbit-stabilizer theorem, this says that the stabilizer of $xF_d$ is a finite index subgroup of $F_d$. Note that $f \in F_d$ stabilizes the left coset $xF_d$ if and only if $x^{-1}fx \in F_d$. The stabilizer being of finite index implies there exists an arbitrarily large integer $m \in \mathbb{N}$ such that $f^m$ stabilizes $xF_d$ for every $f \in F_d$, or equivalently $x^{-1} f^m x \in F_d$ for every $f \in F_d$. Because the $d$-ary cloning system is diverse, we can choose $n_0 \in \mathbb{N}$ such that $$\bigcap_{k=1}^{n} \image \kappa_k^n = \{1\}$$ for all $n \ge n_0$. Now choose $m \in \mathbb{N}$ large enough such that $(m-1)(d-1) \ge n_0$ and such that $x^{-1}f^m x \in F_d$ for every $f \in F_d$. 

By performing a sufficient number of expansions of arbitrary type on the triple representing $x$, if necessary, we can find $d$-ary trees $T$ and $U$ with $n$ leaves such that $n \ge 2+ (m-1)(d-1)$ and $x = [T,g,U]$ for some $g \in G_n$. For $1 < \ell < n$, using Lemma~\ref{powers} observe that $$[T_1,T_n]^{m} = [(\dots(T_1)_1\dots)_1,(\dots(T_n)_{n+d-1}\dots)_{n+(m-1)(d-1)}]$$ and $$[T_\ell,T_n]^{-m} = [(\dots(T_\ell)_\ell\dots)_\ell,(\dots(T_n)_{n+d-1}\dots)_{n+(m-1)(d-1)}]^{-1}.$$ Given how we chose $m \in \mathbb{N}$, we know both of these elements must stabilize the element $x$, and hence it follows that their product $$h =  [(\dots(T_1)_1\dots)_1,(\dots(T_\ell)_\ell\dots)_\ell)]$$
must also stabilize the element $x$. As we noted above, this means, in particular, that conjugating this element of $F_d$ by $x$ gives us an element of $F_d$. On the one hand, note that $x$ can be expanded as $$x = [(\dots(T_\ell)_\ell\dots)_\ell, (g)( \kappa_{k_1}^{n} \circ \kappa_{k_2}^{n+d-1} \circ \dots \circ \kappa_{k_m}^{n+(m-1)(d-1)}), (\dots(U_{k_1})_{k_2}\dots)_{k_m}]$$
where $k_i$ is chosen recursively so that $\rho_n(g)k_1 = \ell$ and $$\rho_{n+(i-1)(d-1)}((g) (\kappa_{k_1}^{n} \circ \kappa_{k_2}^{n+d-1} \circ \dots \circ \kappa_{k_{i-1}}^{n+(i-2)(d-1)}) )k_i = \ell$$ for $i=2,\dots,m$. On the other hand, note that $x^{-1}$ can be expanded as $$x^{-1} = [U,g^{-1},T] = [U', (g^{-1})(\kappa_{1}^{n} \circ \kappa_{1}^{n+d-1} \circ \dots \circ \kappa_{1}^{n+(m-1)(d-1)}), (\dots(T_1)_1\dots)_1],$$ where $U'$ is obtained from $U$ by attaching $d$-ary carets at the appropriate spots. After computing $x^{-1}hx$, we obtain $$[U', (g^{-1})(\kappa_{1}^{n} \circ \kappa_{1}^{n+d-1} \circ \dots \circ \kappa_{1}^{n+(m-1)(d-1)}) \cdot (g)( \kappa_{k_1}^{n} \circ \kappa_{k_2}^{n+d-1} \circ \dots \circ \kappa_{k_m}^{n+(m-1)(d-1)}), (\dots(U_{k_1})_{k_2}\dots)_{k_m}]$$
and this lies in $F_d$ if and only if the group element equals the identity, which means $$(g^{-1})(\kappa_{1}^{n} \circ \kappa_{1}^{n+d-1} \circ \dots \circ \kappa_{1}^{n+(m-1)(d-1)}) \cdot (g)( \kappa_{k_1}^{n} \circ \kappa_{k_2}^{n+d-1} \circ \dots \circ \kappa_{k_m}^{n+(m-1)(d-1)}) = 1$$
or, after rearranging, $$(g^{-1})(\kappa_{1}^{n} \circ \kappa_{1}^{n+d-1} \circ \dots \circ \kappa_{1}^{n+(m-1)(d-1)}) = [(g)( \kappa_{k_1}^{n} \circ \kappa_{k_2}^{n+d-1} \circ \dots \circ \kappa_{k_m}^{n+(m-1)(d-1)})]^{-1}.$$ Using Lemma~\ref{cloning_inverse} to propagate the inverse through the composition of the $d$-ary cloning maps on the right-hand side, the equation becomes $$(g^{-1})(\kappa_{1}^{n} \circ \kappa_{1}^{n+d-1} \circ \dots \circ \kappa_{1}^{n+(m-1)(d-1)}) = [(g^{-1})( \kappa_{\alpha_1}^{n} \circ \kappa_{\alpha_2}^{n+d-1} \circ \dots \circ \kappa_{\alpha_m}^{n+(m-1)(d-1)})],$$
where $\alpha_1 = \rho_n(g) k_1$ and $$\alpha_i = \rho_{n+(i-1)(d-1)}((g)(\kappa_{k_1}^n \circ \kappa_{k_2}^{n+d-1} \circ \dots \circ \kappa_{k_{i-1}}^{n+(i-2)(d-1)}))k_i.$$
But given how we chose the $k_i$, $\alpha_i$ reduces to $\alpha_i = \ell$ for every $i=1,2,\dots,m$, so the equation reduces even further to $$(g^{-1})(\kappa_{1}^{n} \circ \kappa_{1}^{n+d-1} \circ \dots \circ \kappa_{1}^{n+(m-1)(d-1)}) = (g^{-1})( \kappa_{\ell}^{n} \circ \kappa_{\ell}^{n+d-1} \circ \dots \circ \kappa_{\ell}^{n+(m-1)(d-1)})$$ Because $1 < \ell < n$ was arbitrary, it follows that the left-hand side lies in the image of $\kappa_{k}^{n+(m-1)(d-1)}$ for all $1 \le k \le n-1$. All that remains is to show that it lies in the image of $\kappa_k^{n+(m-1)(d-1)}$ for $n \le k \le n+(m-1)(d-1)$, and then the diversity assumption will allow us to conclude $g=1$. To do this, let $1 < \ell \le n$ be arbitrary and consider the element $f = [T_1,T_\ell]$ in $F_d$. Then by Lemma~\ref{powers} $$f^m = [\underbrace{(\dots(T_1)_1\dots)_1}_{(m-1) \text{times}},(\dots(T_\ell)_{\ell+d-1}\dots)_{\ell+(m-1)(d-1)}]$$
and given how we chose $m \in \mathbb{N}$, it must centralize $x$. In this case, the triple representing $x$ can be expanded so that $$x = [(\dots(T_\ell)_{\ell+d-1}\dots)_{\ell + (m-1)(d-1)}, (g)(\kappa_{k_1}^n \circ \kappa_{k_2}^{n+d-1} \circ \dots \circ \kappa_{k_m}^{n+(m-1)(d-2)}), (\dots(U_{k_1})_{k_2}\dots)_{k_m}],$$
where $k_i$ is recursively chosen so that $\rho_n((g)) k_1 = \ell$ and $$\rho_{n+(i-1)(d-1)}((g)(\kappa_{k_1}^n \circ \kappa_{k_2}^{n+d-1} \circ \dots \circ \kappa_{k_{i-1}}^{n + (i-2)(d-1)}))k_i = \ell + (i-1)(d-1)$$
for $i=2,\dots,m$. Computing $x^{-1}f^mx$, we obtain $$[U', (g^{-1}) (\kappa_1^n \circ \kappa_1^{n+d-1} \circ \dots \circ \kappa_1^{n+(m-1)(d-1)}) \cdot (g) (\kappa_{k_1}^n \circ \kappa_{k_2}^{n+d-1} \circ \dots \circ \kappa_{k_m}^{n+(m-1)(d-1)}), (\dots(U_{k_1})_{k_2}\dots)_{k_m}]$$
for some $d$-ry tree $U'$ obtained $U$, which lies in $F_d$ if and only if $$(g^{-1}) (\kappa_1^n \circ \kappa_1^{n+d-1} \circ \dots \circ \kappa_1^{n+(m-1)(d-1)}) \cdot (g) (\kappa_{k_1}^n \circ \kappa_{k_2}^{n+d-1} \circ \dots \circ \kappa_{k_m}^{n+(m-1)(d-1)}) = 1$$
or, after rearranging, $$(g^{-1}) (\kappa_1^n \circ \kappa_1^{n+d-1} \circ \dots \circ \kappa_1^{n+(m-1)(d-1)}) = [(g) (\kappa_{k_1}^n \circ \kappa_{k_2}^{n+d-1} \circ \dots \circ \kappa_{k_m}^{n+(m-1)(d-1)})]^{-1}.$$ Again, using Lemma~\ref{powers} to propagate the inverse through the composition of the $d$-ary cloning maps on the right-hand side, the equation becomes $$(g^{-1}) (\kappa_1^n \circ \kappa_1^{n+d-1} \circ \dots \circ \kappa_1^{n+(m-1)(d-1)}) = (g^{-1}) (\kappa_{\alpha_1}^n \circ \kappa_{\alpha_2}^{n+d-1} \circ \dots \circ \kappa_{\alpha_m}^{n+(m-1)(d-1)})$$
where, given how we recursively chose the $k_i$, $\alpha_1 = \ell$ and $\alpha_i = \ell + (i-1)(d-1)$ for $i=2,\dots,m$. Hence, the above equation reduces even further to $$(g^{-1}) (\kappa_1^n \circ \kappa_1^{n+d-1} \circ \dots \circ \kappa_1^{n+(m-1)(d-1)}) = (g^{-1}) (\kappa_{\ell}^n \circ \kappa_{\ell+d-1}^{n+d-1} \circ \dots \circ \kappa_{\ell + (m-1)(d-1)}^{n+(m-1)(d-1)})$$
This demonstrates that the left-hand side lies in the image of $\kappa_k^{n+(m-1)(d-1)}$ for $n \le k \le n +(m-1)(d-1)$. Hence, combining both parts, we have that $$(g^{-1}) (\kappa_1^n \circ \kappa_1^{n+d-1} \circ \dots \circ \kappa_1^{n+(m-1)(d-1)}) \in \bigcap_{k=1}^{n+(m-1)(d-1)} \image \kappa_{k}^{n+(m-1)(d-1)} = \{1\}$$
and injectivity of the $d$-ary cloning maps tells us that $g=1$. Hence, $x = [T,g,U] = [T,1,U]=[T,U]$ and therefore $x$ belongs to $F_d$, which concludes the proof and hence $L(F_d) \subseteq L(\mathscr{T}_d(G_*))$ satisfies the weak asymptotic homomorphism property. 
\end{proof}

The fact that the inclusion $L(F_d) \subseteq L(\mathscr{T}_d(G_*))$ satisfies the WAHP whenever $(G_n)_{n \in \mathbb{N}}$ is a sequence of groups equipped with a diverse $d$-ary cloning system will have a number of consequences. Some immediate consequences are that normalizers and one-sided (quasi) normalizers of $L(F_d)$ in $L(\mathscr{T}_d(G_*))$ all belong to $L(F_d)$ and hence the inclusion $L(F_d) \subseteq L(\mathscr{T}_d(G_*))$ is singular.

\begin{corollary}
Let $((G_n)_{n \in \mathbb{N}}, (\rho_n)_{n \in \mathbb{N}}, (\kappa_k^n)_{k \le n})$ be a diverse $d$-ary cloning system. Then $\mathcal{QN}^{(1)}_{L(\mathscr{T}_d(G_*))}(L(F_d)) \subseteq L(F_d)$, and also $$ \mathcal{ON}_{L(\mathscr{T}_d(G_*))}(L(F_d))$$ and therefore $$\mathcal{N}_{L(\mathscr{T}_d(G))}(L(F_d))$$ equal $\mathcal{U}(L(F_d))$.
In particular, the inclusion $L(F_d) \subseteq L(\mathscr{T}_d(G_*))$ is singular. 
\end{corollary}

Another consequence of the inclusion $L(F_d) \subseteq L(\mathscr{T}_d(G_*))$ satisfying the WAHP is that the inclusion $L(F_d) \subseteq L(\mathscr{T}_d(G_*))$ is irreducible, which will in turn have a number of consequences which we will see in this section and Section \ref{singularnotWAHP} and Section \ref{sec:higmanmcduff}. As an immediate consequence, $L(\mathscr{T}_d(G_*))$ will be a type $\II_1$ factor whenever $(G_n)_{n \in \mathbb{N}}$ is a sequence of groups equipped with a diverse $d$-ary cloning system. Before we can prove this, however, we need the following lemma.

\begin{lemma}\label{lem:FdICC}
Let $W_d$ be a subgroup of the Higman--Thompson group $V_d$ containing the commutator subgroup $[F_d,F_d]$. Then each non-trivial element of $W_d$ has infinitely many $[F_d,F_d]$-conjugates (so in particular $W_d$ is ICC) which is equivalent to saying the inclusion $L([F_d,F_d]) \subseteq L(W_d)$ is irreducible. In particular, $L([F_d,F_d]) \subseteq L(F_d)$ is irreducible.
\end{lemma}
\begin{proof}
In this proof, we are going to utilize the natural action of $V_d$ on $[0,1)$ by certain right-continuous bijections, and actually the subgroup $F_d$ acts on $[0,1)$ by certain homeomorphism which fix $0$ (see \cite{cannon96}).
Note that a non-trivial $f \in W_d$ has infinitely many $[F_d,F_d]$-conjugates if and only if the centralizer of $f$ in $[F_d,F_d]$ has infinite index in $[F_d,F_d]$. Hence, we will assume that $f \in F_d$ has finite index centralizer in $[F_d,F_d]$ and then argue that it must be the identity. Because $[F_d,F_d]$ is an infinite simple group (see \cite[Theorem~4.13]{brown87}), and since infinite simple groups do not have proper finite index subgroups, it follows that the centralizer of $f$ in $[F_d,F_d$ equals all of $[F_d,F_d]$, meaning that every commutator commutes with $f$. Note that if two elements $g$ and $h$ in $V_d$ commute, then $g$ must setwise stabilize the fixed point set of $h$ (and vice-versa). Given any $a \in (0,1) \cap \mathbb{Z} \left[\frac{1}{d} \right]$, we can select $b \in (0,1) \cap \mathbb{Z} \left[\frac{1}{d} \right]$ and $g \in [F_d,F_d]$ such that the support of $g$ in $(0,1)$ (set of non-fixed points) equals $(a,b)$. Because $g \in [F_d,F_d]$, $f$ and $g$ must commmute, and because $g$ is a right-continuous bijection, we know $g$ must fix $a$. By density of $(0,1) \cap \mathbb{Z}\left[\frac{1}{d}\right]$ in $[0,1)$, we can see that $f=1$.

Because $[F_d,F_d]$ is an infinite simple group, it is necessarily ICC and hence $L([F_d,F_d])$ is a $\II_1$ factor. Hence, by Citation \ref{SmithIrred}, it follows that $L([F_d,F_d]) \subseteq L(W_d)$ is irreducible and hence $L(W_d)$ is a $\II_1$ factor. 
\end{proof}

We note that, technically, Picoaroaga already proved that $F_d$ is ICC (see \cite[Theorem 3.1]{picioroaga06}). However, we included the result for a few reasons. Firstly, it keeps things more self-contained. Secondly, the proof above result $F_d$ is a bit more streamlined than the one in \cite{picioroaga06}. Thirdly, we actually prove something strictly stronger, because the elements Picioroaga defines in \cite{picioroaga06} to prove the ICC property do not all belong to the commutator subgroup. We also note that subgroups of $V_d$ containing $[F_d,F_d]$ deserve to be called Thompson-like, although they generally will not arise from a $d$-ary cloning system (especially if they do not contain $F_d$).

Now we can finally prove that diverse $d$-ary cloning systems give rise to irreducible inclusions and hence these Thompson-like groups yield type $\II_1$ factors:

\begin{corollary}\label{thrm:Irreducible}
Let $((G_n)_{n \in \mathbb{N}}, (\rho_n)_{n \in \mathbb{N}}, (\kappa_k^n)_{k \le n})$ be a $d$-ary cloning system which is diverse. Then the inclusion $L(F_d) \subseteq L(\mathscr{T}_d(G_*)$ is irreducible and in particular $L(\mathscr{T}_d(G_*))$ is a type $\II_1$ factor. 
\end{corollary}
\begin{proof}
Because the sequence of groups $(G_n)_{n \in \mathbb{N}}$ is equipped with a $d$-ary cloning system is diverse, by Theorem \ref{thrm:WAHP} we know that the inclusion $L(F_d) \subseteq L(\mathscr{T}_d(G_*))$ satisfies the WAHP. This entails that the inclusion is also irreducible. By the Lemma \ref{lem:FdICC}, we know that $L(F_d)$ is a type $\II_1$ factor and therefore $L(\mathscr{T}_d(G_*))$ must also be a type $\II_1$ factor.
\end{proof}

A few remarks are in order concerning Theorem \ref{thrm:WAHP} and Theorem \ref{thrm:Irreducible}.  First, from a theoretical standpoint, this is a spectacular improvement of Citation~\ref{factor}. Indeed, we were able to dispense with the fully compatible assumption, some lemmas in \cite{bashwinger}, and some rather technical results in \cite{preaux13}. Moreover, we obtained a much stronger conclusion than merely $L(\mathscr{T}_d(G_*))$ being a type $\II_1$ factor. Since virtually all the important and relevant $d$-ary cloning systems are diverse, we have many examples of inclusions of type $\II_1$ factors which satisfy the WAHP and hence are singular and irreducible, and hence many examples of type $\II_1$ factors (refer to Table \ref{tab:inclusions}). Some of the R{\"o}ver--Nekrashevych groups arise from diverse $d$-ary cloning systems, but since not all them arise from a divese $d$-ary cloning system, we provided an independent argument that they are ICC and hence yield type $\II_1$ factors (see \cite[Proposition 4.7]{bashwinger}), and we shall independently argue that $L(F_d)$ is an irreducible subfactor in the R{\"o}ver--Nekrashevych group factors. 

As a matter of fact, in the case of the R{\"o}ver--Nekrashevych groups we will prove something even stronger---namely, that $L([F_d,F_d])$ is an irreducible subfactor in the R{\"o}ver--Nekrashevych group factors. This fact, together with Lemma \ref{lem:FdICC}, raises the question about whether it is possible to further strengthen Theorem~\ref{thrm:Irreducible} or even prove that the inclusion $L([F_d,F_d]) \subseteq L(\mathscr{T}_d(G_*))$ is singular or, stronger than that, satisfies the WAHP whenever $(G_n)_{n \in \mathbb{N}}$ is equipped with a diverse $d$-ary cloning system. We suspect that the answers to these questions is affirmative, but we presently do not know how to prove it. The difficulty in proving any of these conjectures lies in the substantially complicated tree-theoretic description of the commutator subgroup $[F_d,F_d]$. As a matter of fact, a $d$-ary tree pair $[T,U] \in F_d$ lies in $[F_d,F_d]$ if and only if it lies in the intersection of the kernels of the characters defined in Section 2.3 of \cite{zaremsky17F_n}. In this context, a character on a group $G$ means a group homomorphism $G \to \mathbb{R}$, rather than the usual notion encountered in operator algebras or representation theory which is a normalized, positive definite, tracial function $G \to \mathbb{C}$ (see Section \ref{sec:higmanmcduff} for a more precise definition of a character in the operator-algebraic or representation-theoretic sense). 
Now, a $d$-ary tree pair $[T,U]$ being in the intersection of the kernels of the characters as they are defined in \cite{zaremsky17F_n} amounts to $[T,U]$ satisfying a very complicated equality involving specific measurements on each tree roughly of the form, ``how much does the depth change from the $k$-th leaf to the $(k+1)$-th leaf, summed over all $k$ within a given congruence class $\mod d-1$?", where ``depth" of the $k$-th leaf in a $d$-ary tree is the length of the unique reduced path from the root to the $k$-th leaf. Strengthening Theorem~\ref{thrm:Irreducible} or proving the inclusion $L([F_d,F_d]) \subseteq L(\mathscr{T}_d(G_*))$ is singular or even satisfies the WAHP in the manner just described will require confronting a combinatorial nightmare. We leave it to future work to determine whether these are possible. 

Finally, the fact that the inclusion $L(F_d) \subseteq L(\mathscr{T}_d(G_*))$ is irreducible allows us to conclude that the so-called relative central sequence algebra is non-trivial and even a non-atomic von Neumann algebra. Recall that a type $\II_1$ factor $N$ being McDuff can be alternatively characterized in terms of the central sequence algebra $N' \cap N^{\omega}$ being non-commutative for some free ultrafilter $\omega$ on $\mathbb{N}$, where $N^{\omega}$ is the ultrapower of $N$. If $N$ is of a type $\II_1$ subfactor (not necessarily McDuff anymore) of $M$, we can form the so-called relative central sequence algebra $N' \cap M^\omega$, and clearly we have the inclusion $N' \cap N^{\omega} \subseteq N' \cap M^{\omega}$. In \cite{fang06}, Fang, Ge, and Li proved that if $N \subseteq M$ is an irreducible inclusion of type $\II_1$ factors, then the relative central sequence algebra $N' \cap M$ is either trivial or non-atomic. Hence, if $(G_n)_{n \in \mathbb{N}}$ is equipped with a diverse $d$-ary cloning system, then because $L(F_d)$ is a McDuff factor and because we have the inclusion $L(F_d)' \cap L(F_d)^{\omega} \subseteq L(F_d)' \cap L(\mathscr{T}_d(G_*))^\omega$, we can conclude that the relative central sequence algebra $L(F_d)' \cap L(\mathscr{T}_d(G_*))^\omega$ is non-trivial and hence a non-atomic von Neumann algebra, which we record as the following observation. 

\begin{observation}
Let $((G_n)_{n \in \mathbb{N}}, (\rho_n)_{n \in \mathbb{N}}, (\kappa_k^n)_{k \le n})$ be a $d$-ary cloning system which is diverse. Then the relative central sequence algebra $L(F_d)' \cap L(\mathscr{T}_d(G_*))^{\omega}$ is non-trivial and hence a non-atomic von Neumann algebra. 
\end{observation}

More generally, if the inclusion is irreducible, then the above will hold. We mention this because later in the paper (section?) we will see an example of a cloning system which is not diverse yet the inclusion is irreducible (in fact, it will satisfy the WAHP). We the leave the further analysis of (relative) central sequence algebras arising from $d$-ary cloning systems to future work.

\subsection{R{\"o}ver--Nekrashevych Groups}\label{nekrashevych}

The R{\"o}ver--Nekrashevych groups were first fully introduced by Nekrashevych in \cite{nekrashevych04} as a certain subgroup of the unitary group of a Cuntz-Pimsner algebra. They are defined as a certain ``mash-up" of a self-similar group $G$ and the Higman--Thompson group $V_d$, and typically denoted as $V_d(G)$. In this section, we prove irreducibility for the case of the R{\"o}ver--Nekrashevych group factors, and in fact we obtain a stronger conclusion than in Theorem~\ref{thrm:Irreducible}.  We prove that the inclusion $L([F_d,F_d]) \subseteq L(V_d(G))$ is irreducible, and obtain as an easy corollary that the inclusion $L(F_d) \subseteq L(V_d(G))$ is also irreducible. However, we shall see that neither $L(F_d) \subseteq L(V_d(G))$ nor $L([F_d,F_d]) \subseteq L(V_d(G))$ always satisfies the WAHP.

Although the R{\"o}ver--Nekrashevych groups do in fact arise from $d$-ary cloning systems, as was proven in \cite{skipper21}, we will use neither this description nor their description as a certain group of unitaries in this section. Instead, we use their alternative description as certain homeomorphisms of the boundary of the infinite, rooted, regular  $d$-ary tree $\mathcal{T}_d$. Note that the boundary of $\mathcal{T}_d$ is homeomorphic to the $d$-ary Cantor space $$\mathcal{C}_d := \{1,2\dots,d\}^{\mathbb{N}},$$ which can be identified with the space of all infinite words in the alphabet $\{1,2,\dots,d\}$. This particular model of R{\"o}ver--Nekrashevych groups as certain groups of homeomorphisms is what will enable us to prove some stronger results. As we hinted at in the introduction, the reason for dealing with the R{\"o}ver--Nekrashevych groups separately in this way is that they are the most important class of groups arising from $d$-ary cloning systems which are almost never fully compatible and not necessarily diverse, so none of our current theorems apply to prove that the inclusion is irreducible. Let us now recall the concept of a self-similar group. 

To do this, we need to first properly define $\mathcal{T}_d$. The vertex set of $\mathcal{T}_d$ is $\{1,2,\dots,d\}^*$, the set of all the finite words in the alphabet $\{1,2,\dots,d\}$, with the empty word $\emptyset$ corresponding to the root. Two vertices are adjacent if they are of the form $w$ and $wi$ for $w \in \{1,2,\dots,d\}^*$ and $i \in \{1,2,\dots,d\}$. For each $1 \le i \le d$, define $\mathcal{T}_d(i)$ to be the induced subgraph of $\mathcal{T}_d$ spanned by $iw$ with $w \in \{1,2,\dots,d\}^*$, which is naturally isomorphic to $\mathcal{T}_d$ via $\delta_i : \mathcal{T}_d \to \mathcal{T}_d(i)$ sending $w \mapsto iw$. 

Let $\aut (\mathcal{T}_d)$ be the group of automorphisms of $\mathcal{T}_d$. Because the root is the only vertex of degree $d$, the others having degree $d+1$, every automorphism stabilizes the ``level-$1$ set" of vertices $\{1,2,\dots,d\}$. In particular, we get an epimomorphism $\rho : \aut (\mathcal{T}_d) \to S_d$ with the kernel being the subgroup of all automorphisms fixing every level-$1$ vertex, meaning that it is isomorphic to $\aut (\mathcal{T}_d)^d$. Since the map clearly splits we have $$\aut (\mathcal{T}_d) \cong S_d \ltimes \aut (\mathcal{T}_d)^d,$$
or, more concisely, we have the wreath product decomposition $$\aut (\mathcal{T}_d) \cong S_d \wr \aut (\mathcal{T}_d).$$

For each $1 \le i \le d$, let $\phi_i : \aut (\mathcal{T}_d) \to \aut (\mathcal{T}_d)$ be the function (not homomrphism) given by  $$\phi_i(g) := \delta^{-1}_{\rho(g)i} \circ g \big|_{\mathcal{T}_d(i)} \circ \delta_i,$$
which is well-defined since the image of $g \big|_{\mathcal{T}_d(i)} \circ \delta_i$ is $\mathcal{T}_d(\rho(g)i)$. Finally, a group $G \le \aut (\mathcal{T}_d)$ is said to be \emph{self-similar} if for all $1 \le i \le d$ we have $\phi_i (G) \subseteq G$.

With this in mind, we can now define the R{\"o}ver--Nekrashevych groups. First, given $w \in \{1,2,\dots,d\}^*$, the \emph{cone} on $w$ is defined as $$\mathcal{C}_d(w) := \{w \kappa : \kappa \in C_d\}$$
which is canonically homeomorphic to $\mathcal{C}_d$ via $h_w : \mathcal{C}_d \to \mathcal{C}_d(w)$ given by $\kappa \mapsto w \kappa$. These form a clopen basis for the topology on $\mathcal{C}_d$. Given a self-similar group $G$, the \emph{R{\"o}ver--Nekrashevych group} on $G$ or built from $G$ is the group of all self-homeomorphisms of $\mathcal{C}_d$ defined as follows 
\begin{enumerate}
    \item Take a partition of $\mathcal{C}_d$ into finitely many cones $\mathcal{C}_d(w_1^+)$,\dots,$\mathcal{C}_d(w_n^+)$
    \item Take another partition of $\mathcal{C}_d$ into the same number of cones $\mathcal{C}_d(w_1^-)$,\dots,$\mathcal{C}_d(w_n^-)$
    \item Map $\mathcal{C}_d$ to itself bijectively by sending each $\mathcal{C}_d(w_i^+)$ to some $\mathcal{C}_d(w_j^-)$ via $h_{w_j^-} \circ g_i \circ h_{w_i^+}^{-1}$ for some $g_1,\dots,g_n \in G$.
\end{enumerate}
Note that the composition $h_{w_j^-} \circ g_i \circ h_{w_i^+}^{-1}$ makes sense because we can view $g_i$ as a homeomorphism of $\mathcal{C}_d$, as every automorphism of $\mathcal{T}_d$ induces a homeomorphism of its boundary $\partial \mathcal{T}_d \cong \mathcal{C}_d$. Also, when $G = \{1\}$, we obtain the Higman--Thompson group $V_d(\{1\}) = V_d$. 

First, let us explain why the inclusions $L([F_d,F_d]) \subseteq L(V_d(G))$ and $L(F_d) \subseteq L(V_d(G))$ do not always satisfy the WAHP, and then we will prove irreducibility. Note, any permutation of $\{1,2,\dots , d\}$ naturally induces an automorphism of $\mathcal{T}_d$. The permutation $i \mapsto d-i+1$ induces the ``full reflection" $h_0$ of $\mathcal{T}_d$, and this in turn induces a homeomorphism $h : \mathcal{C}_d \to \mathcal{C}_d$. If $G$ is any self-similar group containing the full reflection automorphism $h_0$ of $\mathcal{T}_d$, then $h$ is a element of $V_d(G)$ which does not belong to $V_d$ (and hence not $F_d$) but which normalizes $F_d$. First, it does not belong to $V_d$ because any element $g \in V_d$ has the property that $\kappa$ and $g(\kappa)$ have the same infinite tail for any $\kappa \in \mathcal{C}_d$; i.e., they may have a different finite prefix but eventually the infinite words $\kappa$ and $g(\kappa)$ ``look" the same. But $h$ violates this property many times over. For example, $h(111\cdots) = ddd\cdots$. Second, it is clear that $h$ normalizes $F_d$ since $F_d$ is precisely the subgroup of $V_d(G)$ which preserves the lexiographic ordering on $\mathcal{C}_d$ (which is in fact a total order), and it is clear that $h$ is order reversing. This shows that the normalizer of $F_d$ in $V_d(G)$ is non-trivial which translates to the inclusion $L(F_d) \subseteq L(V_d(G))$ not being singular and hence not satisfying the WAHP. Since $h^{-1} [f,g] h = [h^{-1}fh, h^{-1}gh]$ for all $f,g \in F_d$, it follows that $h$ also normalizes $[F_d,F_d]$ and hence the inclusion $L([F_d,F_d]) \subseteq L(V_d(G))$ does not satisfy the WAHP for precisely the same reason. Again, if $G$ is any self-similar group containing the full reflection, then neither inclusion satisfies the WAHP.

Since the inclusion $L(F_d) \subseteq L(V_d(G))$ does not satisfy the WAHP, it follows that the $d$-ary cloning system giving rise to these particular R{\"o}ver--Nekrashevych groups cannot be diverse. However, this can be seen more directly: when $d$ is even and $n$ is a power of $d$, $$(1 ~ n)(2 ~ n-1) \dots (d/2 ~ (d/2)+1)(h_0, \dots ,h_0)$$ 
lies in the image of every one of the $d$-ary cloning maps; whereas when $d$ is odd $$(1 ~ n)(2 ~ n-1) \dots ((d+1)/2 ~ ((d+1)/2)+1)(h_0, \dots ,h_0)$$
lies in the images.
We refer to \cite{skipper19} for the definition of the $d$-ary cloning maps used to build the R{\"o}ver--Nekrashevych groups. We leave the description of the (one-sided/quasi-)normalizers of $F_d$ and $[F_d,F_d]$ in $V_d(G)$, and hence their corresponding von Neumann algebra (one-sided/quasi-)normalizers, to future work. Now we prove irreducibility of the inclusion $L([F_d,F_d]) \subseteq L(V_d(G))$ for any self-similar group $G$.

\begin{theorem}
Let $G \le \aut (\mathcal{T}_d)$ be any self-similar group. Then every non-trivial element of the R{\"o}ver--Nekrashevych group $V_d(G)$ has infinitely many $[F_d,F_d]$-conjugates which is equivalent to saying that the inclusion $L([F_d,F_d]) \subseteq L(V_d(G_*))$ is irreducible.
\end{theorem}
\begin{proof}

Suppose that $f \in V_d(G)$ has finite index centralizer in $[F_d,F_d]$. We want to argue that $f=1$. Because $[F_d,F_d]$ is an infinite simple group, it follows that the centralizer of $f$ in $[F_d,F_d]$ must in fact equal $[F_d,F_d]$, meaning that $f$ commutes with every element of $[F_d,F_d]$. Since $f$ commutes with every element of $[F_d,F_d]$, it stabilizes the support of every element of $[F_d,F_d]$. By way of contradiction, suppose $f \neq 1$, meaning there is a point $\kappa \in C_d$ such that $f(\kappa ) \neq \kappa$. Because $\kappa$ and $f(\kappa)$ are distinct, we can choose an open neighborhood $U$ of $\kappa$ which does not contain $f(\kappa)$. Now choose any element $g \in [F_d,F_d]$ such that $\kappa \in \text{supp} (g) \subseteq U$. Since $f$ and $g$ commute, it follows that $$f(\kappa ) \in f(\text{supp}(g)) \subseteq \text{supp}(g) \subseteq U,$$ which is a contradiction. Hence, $f$ must be the identity, thereby establishing the conclusion. 
\end{proof}

From this we easily obtain the following corollary:

\begin{corollary}
Let $G \le \aut (\mathcal{T}_d)$ be any self-similar group. Then every non-trivial element of the R{\"o}ver--Nekrashevych group $V_d (G)$ has infinitely many $F_d$-conjugates which is equivalent to saying that the inclusion $L(F_d) \subseteq L(\mathscr{T}_d(G_*))$ is irreducible.
\end{corollary}

\subsection{Mixing of $L(F_d)$ in $L(\mathscr{T}_d(G_*))$}\label{mixing}
In Section~\ref{WAHP}, we saw that when $(G_n)_{n \in \mathbb{N}}$ is equipped with a diverse $d$-ary cloning system, the inclusion $L(F_d) \subseteq L(\mathscr{T}_d(G_*))$ satisfies the WAHP which is equivalent to the inclusion being weakly mixing. Given how strong the diversity assumption has proven to be thus far, one might wonder whether the subfactor $L(F_d)$ can exhibit the stronger property of being mixing in the von Neumann algebra of any of these groups coming from $d$-ary cloning systems. Unfortunately, this is where the strength of the diversity hypothesis fails us, and, interestingly, to see this we need to appeal to the uniformity property. What we will see is that $L(F_d)$ is never mixing in $L(\mathscr{T}_d(G_*))$ whenever $(G_n)_{n \in \mathbb{N}}$ is equipped with a uniform $d$-ary cloning system. Since virtually all the relevant $d$-ary cloning systems which are uniform are also diverse, this will show that the diversity assumption does not generally imply that the inclusion $L(F_d) \subseteq L(\mathscr{T}_d(G_*))$ is mixing.  

For our purposes, we only care about inclusions of von Neumann algebras arising from group inclusions. Given an inclusion of countable discrete groups $H \le G$, we can translate this to an inclusion of von Neumann algebras $L(H) \subseteq L(G)$. According to \cite[Theorem~4.3]{cameron13}, the inclusion $L(H) \subseteq L(G)$ is mixing (i.e., $L(H)$ is a mixing subalgebra of $L(G)$) if and only if $g^{-1} H g \cap H$ is a finite group for every $g \in G \setminus H$ (i.e., $H$ is almost malnormal in $G$). When $H$ is torsion-free, such as with $F_d$, mixing is equivalent to requiring that $g^{-1} H g \cap H$ is the identity subgroup for every $g \in G \setminus H$ (i.e., $H$ is malnormal in $G$).

Before we explain why $L(F_d) \subseteq L(\mathscr{T}_d(G_*))$ is never mixing whenever $(G_n)_{n \in \mathbb{N}}$ is equipped with a uniform $d$-ary cloning system, let us recall a key property of uniform $d$-ary cloning systems about commuting elements, which was crucial in the proof of Citation \ref{mcduff}. First, note that vertices of a $d$-ary tree can be naturally labelled by finite words in the alphabet $\{1,2, \dots, d\}$. Given a finite word $v$ in this alphabet, we say two $d$-ary trees $T$ and $U$ \emph{agree away from} $v$ if there exists a $d$-ary tree with a leaf labeled $v$ such that each of $T$ and $U$ can be obtained from this tree by gluing a $d$-ary tree to this leaf. With this in mind, we have the following key property of uniform $d$-ary cloning systems from \cite{bashwinger}.

\begin{refer}{\cite[Lemma 5.3]{bashwinger}}\label{cit:uniform}
Let $((G_n)_{n \in \mathbb{N}}, (\rho_n)_{n \in \mathbb{N}}, (\kappa_k^n)_{k \le n})$ be a uniform $d$-ary cloning system. Let $R_-$ be a $d$-ary tree with $n$ leaves, say with one leaf labeled by $v$. Let $R_{+}$ be another $d$-ary with $n$ leaves, one of which is also labelled $v$. Let $T$ and $U$ be trees that agree away from $v$, with the same number of leaves. Then every element of the form $[R_-,g,R_+]$ commutes with $[T,U]$.  
\end{refer}

Let $(G_n)_{n \in \mathbb{N}}$ be a sequence of non-trivial groups equipped with a uniform $d$-ary cloning system. Using Citation \ref{cit:uniform}, it is easy to produce elements $x \in \mathscr{T}_d(G_*) \setminus F_d$ such that $x^{-1} F_d x \cap F_d$ is non-trivial. For example, let $R$ be any $d$-ary tree with the finite word $v$ representing any leaf of $R$, and let $g \in G_{n(R)}$ be any non-trivial element. Define $T$ to be the $d$-ary tree constructed from $R$ by gluing any $d$-ary tree to the leaf of $R$ labelled $v$, and let $U$ be the $d$-ary tree constructed in the same way but gluing any other $d$-ary tree to $v$. Then by construction $T$ and $U$ agree away from $v$ and together they define a non-trivial element $[T,U]$ of $F_d$. Moreover, $x = [R,g,R]$ is an element of $\mathscr{T}_d(G_*)$  not belonging to $F_d$ which, according to Citation \ref{cit:uniform}, commutes with the element $[T,U]$, implying that $x^{-1} F_d x \cap F_d$ is non-trivial. We summarize the preceding discussion into the following observation:

\begin{observation}\label{obs:mixing}
Let $((G_n)_{n \in \mathbb{N}}, (\rho_n)_{n \in \mathbb{N}}, (\kappa_k^n)_{k \le n})$ be a uniform $d$-ary cloning system with $G_n$ a non-trivial group for every $n \in \mathbb{N}$. Then the inclusion $L(F_d) \subseteq L(\mathscr{T}_d(G_*))$ is not mixing. 
\end{observation}

In addition to being diverse, every one of the groups referenced in Table~\ref{tab:inclusions} arises from a uniform $d$-ary cloning system with the exception of the direct product $d$-ary cloning system with monomorphisms. For these $d$-ary cloning systems to be uniform, we need to assume that monomorphisms $\phi_1,\dots,\phi_d : G \to G$ are the identity. However, if they are all the identity, then the $d$-ary cloning system on $(\prod^n(G))_{n \in \mathbb{N}}$ is not diverse, although the $d$-ary cloning system on $(\Psi^n(G))_{n \in \mathbb{N}}$ is diverse. In summary, the diversity assumption does not imply that the inclusion $L(F_d) \subseteq L(\mathscr{T}_d(G_*))$ is mixing.

As for mixing in the non-uniform case, we do not presently know whether this is possible. The only non-uniform $d$-ary cloning systems we know of are some of the direct product examples with monomorphisms and the $d$-ary cloning systems giving rise to the R{\"o}ver--Nekrashevych groups. In Section \ref{non-diverse}, we consider a non-diverse cloning system which also happens to be non-uniform, where we prove that it gives rise to an irreducible inclusion satisfying the WAHP; however, we show that the inclusion is not mixing. 

Regarding the R{\"o}ver--Nekrashevych groups, we saw that the inclusion $L(F_d) \subseteq L(V_d(G))$, as well as $L([F_d,F_d,]) \subseteq L(V_d(G))$, cannot be weakly mixing whenever $G$ is a self-similar group containing the full reflection, and therefore it cannot be mixing since mixing implies weakly mixing. However, we do not know if the inclusion $L(F_d) \subseteq L(V_d(G))$ can be mixing when $V_d(G)$ arises from diverse $d$-ary cloning system. In this case, $G$ cannot contain the full reflection and the inclusion $L(F_d) \subseteq L(V_d(G))$ is weakly mixing, so mixing is not immediately precluded. We leave it to future work to determine whether mixing is possible among the non-uniform $d$-ary cloning systems. Now we turn to some examples of non-diverse $d$-ary cloning systems and study them with respect to the WAHP. 

\subsection{Non-diverse $d$-ary cloning systems and the weak asymptotic homomorphism property}\label{non-diverse}

In Section \ref{diverse}, we gave a table containing a plethora of examples of diverse $d$-ary cloning systems and hence instances where the inclusion $L(F_d) \subseteq L(\mathscr{T}_d(G_*))$ satisfies the WAHP. In this section, we explore some examples of non-diverse $d$-ary cloning systems with respect to this property. What we will see is that it is possible for the inclusion $L(F_d) \subseteq L(\mathscr{T}_d(G_*))$ to satisfy the WAHP and for it not to, thereby showing that while diversity is an incredibly strong sufficient condition, it is not a necessary condition.

As for a non-diverse $d$-ary cloning system producing an inclusion $L(F_d) \subseteq L(\mathscr{T}_d(G_*))$  does not satisfy the WAHP, we have already seen such an example when we looked at R{\"o}ver--Nekrashevych groups. We saw that as soon as a self-similar group $G$ contains the full reflection described in Section \ref{nekrashevych}, the $d$-ary cloning system cannot be diverse, and that the homeomorphism induced by the full reflection normalizes $F_d$. This entails that the inclusion $L(F_d) \subseteq L(V_d(G))$ is non-singular and hence does not satisfy the WAHP. But on the other hand, the inclusion $L([F_d,F_d]) \subseteq L(V_d(G))$ is irreducible and therefore $L(F_d) \subseteq L(V_d(G))$ is, too. It would be nice to produce an example for which we have irreducibility but not the WAHP using the direct product example with monomorphisms. At the moment, however, it is unclear how to do this.

Let us first look at a non-diverse $d$-ary cloning system for which the inclusion $L(F_d) \subseteq L(\mathscr{T}_d(G_*))$ does not satisfy the WAHP using the direct product examples with monomorphisms.

\begin{example}[Examples not satisfying the WAHP]
With respect to the direct product example with monomorphisms, to find some non-diverse $d$-ary cloning systems we need to consider the case where the images of the monomorphisms $\phi_1,\dots,\phi_d : G \to G$ do intersect non-trivially. In this case, we can achieve this by simply letting all the monomorphisms be the identity. Hence, let $G$ be any non-trivial group (typically ICC so that $\mathscr{T}_d(\prod^*(G))$ is ICC) and let $\phi_i : G \to G$ be the identity morphism for all $i=1,\dots,d$. First, in this setup, it is rather easy to see that the $d$-ary cloning system is in fact non-diverse because the $n$-tuple $(g,\dots,g)$ lies in the image of every $d$-ary cloning map for any $g \in G$ and every $n \in \mathbb{N}$. Second, the inclusion $L(F_d) \subseteq L(\mathscr{T}_d(\prod^*(G))$ is never irreducible. If $g \in G$ is any non-trivial element, then it is not difficult to argue that $[\cdot,g,\cdot]$ is a non-trivial element fixed by $F_d$ via conjugation and hence has finitely many $F_d$-conjugates. Of course, this also shows that the normalizer of $F_d$ in $\mathscr{T}_d(\prod^*(G))$ is non-trivial and therefore the inclusion $L(F_d) \subseteq L(\mathscr{T}_d(\prod^*(G))$ does not satisfy the WAHP. We note that all of this works if instead we just take the monomorphisms $\phi_1,\dots,\phi_d : G \to G$ to simply have a common non-trivial fixed point in $G$.

Let us actually compute the normalizer of $F_d$ in $\mathscr{T}_d(\prod^*(G))$. We claim that this normalizer consists of all elements of the form $$[T,(g,\dots,g),U]$$
with $T$ and $U$ $d$-ary trees and $g \in G$. First, it is clear that any element of this form normalizes $F_d$. As for the other direction, suppose that $[T,(g_1,\dots,g_n),U]$ normalizes $F_d$. Expanding the triple $(T, (g_1,\dots,g_n),U)$ by adding a caret at the first leaf and applying the appropriate cloning map to the tuple $(g_1,g_2,\dots,g_n)$, the first entry $g_1$ is replaced by a ``block" of $g_1$'s of length $d$. Expanding $T$ at the first leaf $m$ times successively and applying the appropriate sequence of cloning maps to the tuple, the first entry $g_1$ is replaced by a block of $g_1$'s of length $md$. Choose $m \in \mathbb{N}$ such that $md \ge n$. Note that, up to expansion, $[T,(g_1,\dots,g_n),U]$ is equivalent to both $$[(\dots((\underbrace{T_1)_1)\dots)_1}_{m \text{ times}},(\underbrace{g_1,\dots,g_1}_{md \text{ times}}, g_2,g_3\dots,g_n), (\dots((U_1)_1)_1\dots)_1],$$
which is formed by adding a $d$-ary caret at the first leaf of the tree $m$-times and $$[(\dots((T_n)_{n+d-1})_{n+2(d-1)}\dots)_{n+m(d-1)}, (g_1,g_2,\dots,\underbrace{g_n,\dots,g_n}_{md \text{ times}}) (\dots((U_n)_{n+d-1})_{n+2(d-1)}\dots)_{n+m(d-1)}],$$
which is formed by adding a $d$-ary caret at the last leaf of the tree $m$-times. Now consider the element $$[(\dots((\underbrace{T_1)_1)_1\dots)_1}_{m \text{ times}}, (\dots((T_n)_{n+d-1})_{n+2(d-1)}\dots)_{n+m(d-1)}]$$
which lies in $F_d$. Since $md \ge n$, when we conjugate this particular element of $F_d$ by $[T,(g_1,\dots,g_n),U]$ using the above two expansions, the whole block of $g_1$'s in the tuple $(g_1,\dots,g_1, g_2,g_3\dots,g_n)^{-1}$ completely ``covers" the block $g_1,g_2,\dots,g_n$ in the tuple $(g_1,g_2,\dots,g_n,\dots,g_n)$. This means that when we compute the conjugate, what we obtain is $$[(\dots((T_1)_1)_1\dots)_1, (1,g_1^{-1}g_2, g_1^{-1}g_3,\dots,g_1^{-1}g_n,\dots),(\dots((U_n)_{n+d-1})_{n+2(d-1)}\dots)_{n+m(d-1)}].$$
Now, the entries following $g_1^{-1}g_n$ in the tuple are immaterial for our purposes, hence our reason putting ellipses; we do not mean to indicate that this is an infinite sequence. Now this lies is in $F_d$ because we assumed that $[T,(g_1,\dots,g_n),U]$ normalizes $F_d$. Hence, it follows that $g_1^{-1}g_2 = 1$, $g_1^{-1} g_3=1$,\dots, $g_1^{-1}g_n = 1$, and therefore $g_1 = g_2 = \dots = g_n$, proving the claim.

Unfortunately, because the inclusion is not irreducible, this means that the normalizer of $L(F_d)$ in $L(\mathscr{T}_d(\prod^*(G)))$ does not admit a nice description in terms of the group normalizer. It is clear that all elements of the form $w \lambda_x$, where $w \in U(L(F_d))$ and $x \in \mathcal{N}_{\mathscr{T}_d(\prod^*(G))}(F_d)$, normalize the subfactor $L(F_d)$. However, without irreducibility the normalizer of $L(F_d)$ in $L(\mathscr{T}_d(\prod^*(G))$ can be strictly larger. We leave it to future work to compute the von Neumann algebra (one-sided/quasi-)normalizer of these, and related, examples.

Now let us turn to some examples of non-diverse $d$-ary cloning systems where we will see that the inclusion satisfies the WAHP. 
\end{example}

\begin{example}[Examples satisfying the weak asymptotic homomorphism property]

Under the case of the images of the monomorphisms $\phi_1,\dots,\phi_d : G \to G$ intersecting non-trivially, we just saw that when all of the monomorphisms are the identity morphism (or they at least have a non-trivial common fixed point), the inclusion does not satisfy the WAHP. Therefore, if we want a non-diverse $d$-ary cloning system to produce an inclusion which satisfies the WAHP, we still need to look among those examples where the images intersect non-trivially but we need to modify things slightly. We achieve this taking one of the monomorphisms to be a non-identity monomorphism with some additional properties, and we do this by considering the classical case (i.e., $d=2$). First, by saying a self-homomorphism $\phi$ of a group $G$ is \emph{fixed-point-free}, we mean that there is no $g \in G \setminus \{1\}$ for which $\phi (g) = g$.

\begin{proposition}\label{prop:fixed-point}
Let $G$ be a non-trivial group admitting a fixed-point-free automorphism $\phi$ of order two. Take $\phi_1$ be the identity map and $\phi_2 = \phi$. Then the cloning system on $(\prod^n(G))_{n \in \mathbb{N}}$ with the monomorphisms $\phi_1, \phi_2$ is not diverse but the inclusion $L(F) \subseteq L(\mathscr{T}(\prod^*(G)))$ satisfies the weak asymptotic homomorphism property. 
\end{proposition}
\begin{proof}
First, it is not difficult to show that when $n$ is odd, the $n+1$ tuple $$(g,\phi(g),g, \dots ,g,\phi(g))$$ 
lies in the images of the cloning maps $\kappa_k^n$ for $1 \le k \le n$ for any $g\in G$,
whereas when $n$ is even, the $(n+1)$-tuple $$(g,\phi(g),g,\dots,\phi(g),g)$$ 
lies in the images for any $g \in G$, thereby proving that the cloning system fails to be diverse. Now let us argue that the inclusion satisfies the WAHP.

Because the cloning system is pure, we get the internal semi-direct decomposition $$\textstyle \mathscr{T}(\prod^*(G))) = \mathscr{K}(\prod^*(G)) \rtimes F.$$
From this decomposition, it is not difficult to see that to prove that the inclusion $L(F) \subseteq L(\mathscr{T}(\prod^*(G)))$ satisfies the WAHP, it is sufficient (and necessary) to show that if $x, x_1,\dots,x_t \in \mathscr{K}(\prod^*(G))$ satisfy $$Fx \subseteq \bigcup_{i=1}^t x_i F,$$
then $x =1$. As we noted in the proof of Theorem~\ref{thrm:WAHP}, the above right coset covering condition implies that the stabilizer of $xF$ under the left-translation action of $F$ on $\mathscr{T}(G^*)/F$ is a finite index subgroup of $F$; moreover, an element $f \in F$ stabilizes $xF$ under this action if and only if $x^{-1}fx \in F$. Since the stabilizer has finite index in $F$, there is an arbitrarily large $m \in \mathbb{N}$ such that $f^m$ stabilizes $xF$ for every $f \in F$. Because $x \in \mathscr{K}(\prod^*(G))$, this means that $x$ is of the form $[T,(g_1, \dots , g_n), T]$ for some binary tree $T$ with $n$ leaves and $g_1, \dots , g_n \in G$. Choose $m \in \mathbb{N}$ large enough such that $m > n+1$ and $f^{m+1}$ stabilizes $xF$ for every $f \in F$. If $f = [T_1,T_n]$, then using Lemma ~\ref{powers}, it is not hard to see that $$f^{m+1} = [T_1,T_n]^{m+1} = [(\dots((\underbrace{T_1)_1)_1\dots)_1}_{m \text{ times}}, (\dots((T_n)_{n+1})_{n+2}\dots)_{n+m}]$$ which stabilizes $x$, meaning that $x^{-1} f^{m+1} x \in F$. On the one hand, we expand the triple representing $x$ by adding carets to the last leaf of the tree $m$-times and apply the corresponding sequence of cloning maps so that the group element becomes $$(g_1,g_2, \dots , g_n, \phi(g_n),g_n, \phi (g_n), \dots, \epsilon),$$ where $\epsilon$ is either $g_n$ or $\phi(g_n)$, depending on whether $m$ is even or odd. On the other hand, we expand the triple representing $x^{-1}$ by adding carets to the first leaf of the tree $m$-times and apply the corresponding sequence of cloning maps so that the group element becomes $$(g_1^{-1},\underbrace{\phi (g_1^{-1}), \dots ,\phi (g_1^{-1})}_{(m-1) \text{ times}}, g^{-1}_2, \dots ,g^{-1}_n).$$ Hence, calculating $x^{-1} f_m x$ and inspecting the group element in the triple representing $x^{-1}f_mx$, we see that it equals $$(1, \phi(g_1^{-1}) g_2,\phi(g_1^{-1})g_3,\dots, \phi(g_1^{-1})g_n, \phi(g_1^{-1}) \phi(g_n),\dots,\epsilon'),$$
where $\epsilon '$ is either $1$ or $g_n^{-1}\phi(g_n)$, depending on whether $m$ is even or odd. Since $x^{-1}f_m x$ belongs to $F$, it follows that the above tuple is the identity element and hence $\phi (g_1) = g_2$, $\phi (g_1) = g_3$,\dots, $\phi(g_1) = g_n$, and $\phi (g_1) = \phi(g_n)$. The last two equations imply $\phi (g_n) = g_n$ and hence $g_n = 1$ because $\phi$ is a fixed-point-free automorphism. From here it follows that $g_1 = g_2 = g_3 = \dots = g_n = 1$, and therefore $x=1$. Hence, the inclusion $L(F) \subseteq L(\mathscr{T}(\prod^*(G)))$ satisfies the WAHP. 
\end{proof}

Fixed-point free automorphisms of order two are fairly easy to come by. For example, if $G$ is any non-trivial abelian group without $2$-torsion, then $\phi$ given by $g \mapsto g^{-1}$ is a fixed-point-free automorphism of order two. As the more interesting examples in von Neumann algebra theory tend to come from the non-amenable realm, let us look to a non-amenable example. Note that $G$ being non-amenable ensures $\mathscr{T}(\prod^*(G))$ is non-amenable, independently of whether $F$ is amenable, which is, arguably, the biggest open problem concerning $F$. Of course, more generally, if $G$ is non-amenable, then neither $\mathscr{T}_d(\prod^*(G))$ nor $\mathscr{T}_d(\Psi^*(G))$ is amenable. For a non-amenable example, consider $G = \mathbb{F}_2$ be a free group freely generated by $a$ and $b$, and let $\phi : \mathbb{F}_2 \to \mathbb{F}_2$ be the homomorphic extension of $a \mapsto b$ and $b \mapsto a$. The homomorphism $\phi$ is certainly an isomorphism. In fact, it is its own inverse so it is certainly of order two. Finally, if $g \in \mathbb{F}_2$ is a non-trivial element, then the first letter of $\phi (g)$ and $g$ will be different, meaning they cannot possibly be equal, whereby we conclude $\phi$ is fixed-point-free. Hence, the inclusion $L(F) \subseteq L(\mathscr{T}(\prod^*(\mathbb{F}_2)))$ satisfies the WAHP.

Note that this cloning system on $(\prod^n(G))_{n \in \mathbb{N}}$, or even restricted to $(\Psi^n(G))_{n \in \mathbb{N}}$, is very far from being uniform (see Definition~\ref{uniform}). Indeed, for $n \in \mathbb{N}$ and $1 < k \le n$, $$(g_1,\dots,g_{k-1},g_k,g_{k+1},\dots,g_n)(\kappa_k^n \circ \kappa_k^{n+1}) = (g_1,\dots,g_{k-1},g_k, \phi(g_k), \phi(g_k),g_{k+1},\dots,g_n)$$
while
$$(g_1,\dots,g_{k-1},g_k,g_{k+1},\dots)(\kappa_k^n \circ \kappa_{k+1}^{n+1}) = (g_1,\dots,g_{k-1},g_k,\phi(g_k),g_k,g_{k+1},\dots,g_n).$$
If they were equal, then, in particular, $\phi(g_k) = g_k$. However, since $\phi$ is fixed-point free, it is clear that these tuples are generally not equal. Recall that the uniformity property was crucial in proving that $L(\mathscr{T}_d(G_*))$ can be a type $\II_1$ McDuff factor (see Citation~\ref{mcduff}). Hence, neither $L(\mathscr{T}(\prod^*(G))$ nor $L(\mathscr{T}(\Psi^*(G))$ is obviously a type $\II_1$ McDuff factor, or at least it is not an immediate consequence of Citation~\ref{mcduff}. We wonder, do either yield a type $\II_1$ factor which is either McDuff or has property Gamma?

Since the uniform property precludes mixing (see Observation~\ref{obs:mixing}), but neither cloning system has the uniform property, one might wonder whether $L(F)$ is a mixing subfactor in these examples. It turns out, however, that we do not have mixing. Figure~\ref{fig:mixing} gives two commuting elements, one which belongs to $\mathscr{T}(\prod^*(G))$ and does not belong to $F$ whenever $g \in G$ is non-trivial, while the other element is a non-trivial element belonging to $F$. The fact that these elements commute implies that the inclusion $L(F) \subseteq L(\mathscr{T}(\prod^*(G)))$ is not mixing. We note that this choice of element actually belongs to $\mathscr{T}(\Psi^*(G))$ and hence shows that the inclusion $L(F) \subseteq L(\mathscr{T}(\Psi^*(G)))$ is not mixing either. Since the cloning system on $(\Psi^n(G))_{n \in \mathbb{N}}$ is diverse, Theorem~\ref{thrm:WAHP} tells us that the inclusion $L(F) \subseteq L(\mathscr{T}(\Psi^*(G))$ is weakly mixing though.

\begin{figure}[htb]
 \centering
\begin{tikzpicture}[line width=0.8pt, scale=0.4]
  \node at (-1.5,-1){$\Big[$};
  \draw (-1,-1) -- (0,0) -- (1,-1);
  \filldraw (-1,-1) circle (1.5pt)   (0,0) circle (1.5pt)   (1,-1) circle (1.5pt);
  \node at (1.5,-1){,};
  \node at (3,-1){$(1,g)$};
  \node at (4.5,-1){,};
   \begin{scope}[xshift=2.35in]
   \draw (-1,-1) -- (0,0) -- (1,-1);
   \filldraw (-1,-1) circle (1.5pt)   (0,0) circle (1.5pt)   (1,-1) circle (1.5pt);
   \node at (1.5,-1){$\Big]$};
   \end{scope}

    \node at (9.5,-1) {and};
   
  \begin{scope}[xshift=6.3in]
   \node at (-4,-1.5){$\Bigg[$};
   \draw  (0,-2) -- (-1,-3)  (0,-2) -- (1,-3)  (-2,-2) -- (-1,-1) (0,-2) -- (-1,-1)  (-1,-1) -- (0,0) -- (2,-2) (-1,-1) -- (-3,-3) ;
   \filldraw (-1,-3) circle (1.5pt) (1,-3) circle (1.5pt)  (-3,-3) circle (1.5pt) (0,-2) circle (1.5pt)    (0,0) circle (1.5pt) (-1,-1) circle (1.5pt)  (2,-2) circle (1.5pt);
   \node at (3.25,-2){,};
   \begin{scope}[xshift=2.5in]
    \draw (-3,-3) -- (0,0) -- (2,-2)   (0,-2) -- (-1,-1)  (-2,-2) -- (-1,-3)   (3,-3) -- (2,-2);
    \filldraw  (-3,-3) circle (1.5pt)  (-2,-2) circle (1.5pt)   (0,0) circle (1.5pt)  circle (1.5pt)  (-1,-1) circle (1.5pt) (-1,-3) circle (1.5pt)  (0,-2) circle (1.5pt)  (3,-3) circle (1.5pt);
    \node at (4,-1.5){$\Bigg]$};
   \end{scope}
  \end{scope}

\end{tikzpicture}
\caption[]{A pair of commuting elements showing that neither $L(F) \subseteq L(\mathscr{T}(\prod^*(G)))$ nor $L(F) \subseteq L(\mathscr{T}(\Psi^*(G)))$ is a mixing inclusion. The leftmost element belongs to $\mathscr{T}(\Psi^*(G)) \le \mathscr{T}(\prod^*(G))$, while the rightmost element belongs to $F$.}
\label{fig:mixing}
\end{figure}
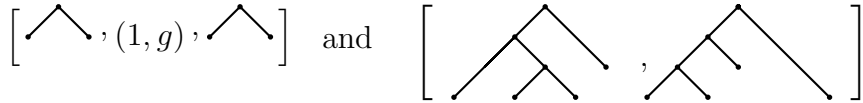

In conclusion, using the direct product examples with monomorphisms, it is quite easy to produce examples of non-diverse $d$-ary cloning systems where the inclusion is still irreducible and satisfies the WAHP.
\end{example}

\section{Singular Inclusions of type $\II_1$ Factors without the Weak Asymptotic Homomorphism Property}\label{singularnotWAHP}

In this section, using $d$-ary cloning systems and amalgamated free products, we construct a variety of singular inclusions of type $\II_1$ factors which do not satisfy the WAHP. As we noted in the introduction, the first example of such an inclusion was constructed in \cite{grossman10}. These do not have the WAHP by virtue of being proper finite index inclusions, which immediately precludes the WAHP. This left open the case of producing an infinite index, singular inclusion of type $\II_1$ factors without the WAHP. In this section, we provide a variety of examples of such infinite index inclusions (and finite index inclusions, too), although we note that these examples were essentially constructed simultaneously with the ones in \cite{bannon23}. However, we note that ours differs in that our constructions are purely group-theoretic. 

Somewhat vaguely, we construct such inclusions by starting with an inclusion of groups with certain desired properties and then ``Thompsonify" the inclusion by passing it through the modified direct product $d$-ary cloning system. We will see that Thompsonifying the inclusion preserves these properties, including the index of the group inclusion. More precisely, given an inclusion $H \le G$ of countable discrete groups, throughout this section we fix $$H_n := \{1\} \times \underbrace{G \times \dots \times G}_{n-2 \text{ times}} \times H$$ and $G_n := \Psi^n(G)$ for $n \ge 3$, and we consider the sequence of groups $(G_n)_{n \in \mathbb{N}}$ with its usual $d$-ary cloning system with all the monomorphisms $\phi_1, \dots , \phi_d : G \to G$ the identity. The $d$-ary cloning system respects the sequence of subgroups $(H_n)_{n \in \mathbb{N}}$, meaning we can restrict the $d$-ary cloning system to $(H_n)_{n \in \mathbb{N}}$ to get a $d$-ary cloning subsystem and from this we obtain the inclusion $F_d \le \mathscr{T}_d(H_*) \le \mathscr{T}_d(G_*)$.  Recall that the $d$-ary cloning system on $(G_n)_{n \in \mathbb{N}}$ is pure, uniform, and diverse and hence so is the $d$-ary cloning system restricted to $(H_n)_{n \in \mathbb{N}}$. Because the $d$-ary cloning system on $(G_n)_{n \in \mathbb{N}}$ is diverse, we know by Theorem \ref{thrm:Irreducible} that the inclusion $L(F_d) \subseteq L(\mathscr{T}_d(G_*))$ is irreducible and hence $L(\mathscr{T}_d(H_*)) \subseteq L(\mathscr{T}_d(G_*))$ must also be irreducible.  This means that the normalizer of $L(\mathscr{T}_d(H_*))$ in $L(\mathscr{T}_d(G_*))$ can be described in terms of the normalizer of $\mathscr{T}_d(H_*)$ in $\mathscr{T}_d(G_*)$ as outlined in Citation \ref{SmithNormalizer}. 

For this particular $d$-ary cloning system we have the following lemma which says that Thompsonifying the inclusion $H \le G$ in the above way preserves the index.

\begin{lemma}\label{lem:index}
Let $H \le G$ be an inclusion of groups, and define $H_n$ and $G_n$ as above equipped with the $d$-ary cloning system described above. Then 

$$|\mathscr{T}_d(G_*) : \mathscr{T}_d(H_*)| = |G:H|.$$
\end{lemma}
\begin{proof}
First, note that 
\begin{align*}
|\mathscr{T}_d(G_*) : \mathscr{T}_d(H_*)| &=  |\mathscr{K}_d(G_*) \rtimes F_d : \mathscr{K}_d(H_*) \rtimes F_d| \\
&= |\mathscr{K}_d(G_*) : \mathscr{K}_d(H_*)|
\end{align*}
Of course, this is true more generally of any pure $d$-ary cloning system on a sequence of groups on $(G_n)_{n \in \mathbb{N}}$ and a sequence of subgroups $(H_n)_{n \in \mathbb{N}}$ which ``respects" the cloning system. Hence, it suffices to argue that the index of $\mathscr{K}_d(H_*)$ in $\mathscr{K}_d(G_*)$ is $|G:H|$. To this end, let $S \subseteq G$ be a set of distinct left-coset representatives of $H$ in $G$. We claim that all the elements of the form $$[\Lambda_d, (\underbrace{1,\dots,1}_{d-1},x),\Lambda_d]~~~~(\ast)$$
for $x \in S$, of which there are $|S| = |G:H|$ of them, form all the distinct representatives for left cosets in $\mathscr{K}_d(G_*)/\mathscr{K}_d(H_*)$, where $\Lambda_d$ is the $d$-ary caret. To see this, let $[T,(1,g_2,\dots,g_n),T]$ represent some arbitrary element in $\mathscr{K}_d(G_*)/\mathscr{K}_d(H_*)$ for some $d$-ary tree $T$ with $n$ leaves and elements $g_2, \dots ,g_n \in G$, and we will argue it is equivalent to something of the form $(\ast)$ up to expansion and modulo $\mathscr{K}_d(H_*)$. Given $g_n \in G$, there exists some left-coset representation $x \in S$ and some $h \in H$ such that $g_n = xh$. Then
\begin{align*}
[T,(1,g_2,\dots,g_n),T] &= [T,(1,g_2, \dots ,g_{n-1},h),T] \\
&= [T,(1,1,\dots,1,x),T][T,(1,g_2,\dots,g_{n-1},h)] \\
\end{align*}
Modulo $\mathscr{K}_d(H_*)$, the element $[T,(1,g_2,\dots,g_n),T]$ is equivalent to $$[T,(1,\dots,1,x),T]$$
in the quotient $\mathscr{K}_d(G_*)/\mathscr{K}_d(H_*)$. Since any $d$-ary tree can be inductively constructed from the $d$-ary caret $\Lambda_d$,
the $d$-ary tree $T$ is obtainable from $\Lambda_d$ via a finite number of appropriate expansions. In general, without knowing precisely what $T$ looks like, we do not know which expansions we need to perform in order to obtain $T$ from $\Lambda_d$. But in this specific case that will pose no issue. We can expand $$[\Lambda_d, (1,\dots,1,x),\Lambda_d]$$
so that the $d$-ary caret $\Lambda_d$ becomes $T$, and this will require applying the appropriate sequence of $d$-ary cloning maps the $d$-tuple $$(1,\dots,1,x)$$ in accordance with the expansions needed to obtain $T$ from $\Lambda_d$. Doing this, the element $[\Lambda_d, (1,\dots,1,x),\Lambda_d]$ is equivalent to $$[T,(1,\dots,1,\underbrace{x,\dots,x},x),T]$$
where the underlined block of $x$'s might be empty depending on whether we need to clone the last coordinate at any point in expanding $\Lambda_d$ to obtain $T$. If the block is empty, then we are finished as we just showed that any element of $\mathscr{K}_d(G_*)/\mathscr{K}_d(H_*)$ is equivalent to an element represented by $(\ast)$ for some $x \in S$ via expansions and reduction modulo $\mathscr{K}_d(H_*)$. However, if the block is \textit{not} empty, then the element can be written as $$[T,(1,\dots,1,\underbrace{x,\dots,x},x),T] = [T,(1,\dots,1,\underbrace{x,\dots,x},1),T][T,(1,\dots,1,x),T]$$
and modding out in the quotient $\mathscr{K}_d(G_*)/\mathscr{K}_d(H_*)$ this is clearly equivalent to $$[T,(1,\dots,1,x),T]$$
Hence, using expansions and reduction modulo $\mathscr{K}_d(H_*)$, this shows that any element of $\mathscr{K}_d(G_*)/\mathscr{K}_d(H_*)$ can be represented by an element of the form $$[\Lambda_d, (\underbrace{1,\dots,1}_{d-1},x_i),\Lambda_d]$$ for some $x \in S$,
of which there are $|S| = |G:H|$ them, and we note that these elements are clearly pairwise distinct. Whence it follows $$|\mathscr{K}_d(G_*) : \mathscr{K}_d(H_*)|= |G:H|$$
\end{proof}

Next, we need to a lemma which allows us describe the group normalizer of $\mathscr{T}_d(H_*)$ in $\mathscr{T}_d(G_*)$.
\begin{lemma}\label{lem:normalizer}
Let $H \le G$ be an inclusion of groups, and define $H_n$ and $G_n$ as above equipped with the $d$-ary cloning system described above. If $x \in \mathscr{T}_d(G_*)$ normalizes $\mathscr{T}_d(H_*)$, then $x = [T,(1,g_2, \dots ,g_n), U]$ for some $d$-ary trees $T$ and $U$ with $n$ leaves and $g_2, \dots ,g_n \in G$ with $g_n \in \mathcal{N}_G(H)$. In particular, if $H$ is self-normalizing in $G$, then $\mathscr{T}_d(H_*)$ is self-normalizing in $\mathscr{T}_d(G_*)$
\end{lemma}
\begin{proof}
Suppose that $x \in \mathscr{T}_d(G_*)$ normalizes $\mathscr{T}_d(H_*)$ and write $x = [T,(1,g_2, \dots ,g_n),U]$ for some $d$-ary trees $T$ and $U$ with $n$ leaves and $g_2, \dots ,g_n \in G$. Let $h \in H$ be arbitrary and consider the element $[T,(1,1, \dots ,1,h),T]$ in $\mathscr{T}_d(H_*)$. Then, because $x$ normalizes $\mathscr{T}_d(H_*)$, we know that 

\begin{align*}
[T,(1,g_2, \dots ,g_n),U]^{-1} [T,(1,1, \dots ,1,h),T] [T,(1,g_2, \dots ,g_n),U] &= [T,(1,g_2, \dots ,g_n^{-1} h g_n), T] \\
\end{align*}
belongs to $\mathscr{T}_d(H_*)$ which means that $g_n^{-1} h g_n \in H$. Reversing the positions of $x$ and $x^{-1}$, we can show that $g_n h g_n^{-1} \in H$. Hence, because $h \in H$ was arbitrary, it follows that $g_n \in \mathcal{N}_G(H)$. 

Finally, if $H$ is self-normalizing in $G$, then $g_n \in \mathcal{N}_G(H) = H$ and therefore $x \in \mathscr{T}_d(H_*)$, which means $\mathscr{T}_d(H_*)$ is self-normalizing in $\mathscr{T}_d(G_*)$ and therefore completes the proof.
\end{proof}
 
Mimicking the calculations above, it is not terribly difficult to show that we can get a similar description of the normalizer of $\mathscr{T}_d(\Psi^*(H))$ in $\mathscr{T}_d(\Psi^*(G))$, or even $\mathscr{T}_d(\prod^*(H))$ in $\mathscr{T}_d(\prod^*(G))$. For example, in the latter case, if $x \in \mathscr{T}_d(\prod^*(G))$ normalizes $\mathscr{T}_d(\prod^*(H))$, then $x = [T,(g_1, \dots ,g_n), U]$ for some $d$-ary trees $T$ and $U$ and $g_1, \dots ,g_n \in \mathcal{N}_G(H)$. 

We immediately obtain the following corollary. 

\begin{corollary}\label{cor:singular}
Let $H \le G$ be an inclusion of groups, and define $H_n$ and $G_n$ as above equipped with the $d$-ary cloning system described above. If $H$ is self-normalizing in $G$, then $L(\mathscr{T}_d(H_*) \subseteq L(\mathscr{T}_d(G_*))$ is a singular inclusion of type $\II_1$ factors. 
\end{corollary}
\begin{proof}
If $H$ is self-normalizing in $G$, then Lemma \ref{lem:normalizer} tells us that $\mathscr{T}_d(H_*)$ is self-normalizing in $\mathscr{T}_d(G_*)$. Because the inclusion $L(\mathscr{T}_d(H_*)) \subseteq L(\mathscr{T}_d(G_*))$ is an irreducible of type $\II_1$ factors, the normalizer of $L(\mathscr{T}_d(H_*))$ in $L(\mathscr{T}_d(G_*))$ can be computed from the normalizer of $\mathscr{T}_d(H_*)$ in $\mathscr{T}_d(G_*)$. As a matter of fact, using Citation \ref{SmithNormalizer}, it is clear that the normalizer of $L(\mathscr{T}_d(H_*))$ in $L(\mathscr{T}_d(G_*))$ is simply $\mathcal{U}(L(\mathscr{T}_d(H_*)))$, which means that the normalizer generates $L(\mathscr{T}_d(H_*))$ and therefore the inclusion is singular. Hence, $L(\mathscr{T}_d(H_*)) \subseteq L(\mathscr{T}_d(G_*))$ is a singular inclusion of type $\II_1$ factors. 
\end{proof}

Again, the above corollary is true if instead you consider the Thompsonifications $\mathscr{T}_d(\Psi^*(H)) \le \mathscr{T}_d(\Psi^*(G))$ or $\mathscr{T}_d(\prod^*(H)) \le \mathscr{T}_d(\prod^*(G))$.

Using Lemma \ref{lem:index} and Corollary \ref{cor:singular}, we can already easily construct singular inclusions of type $\II_1$ factors which do not satisfy the WAHP. Indeed, if $H \le G$ is any finite index inclusion of groups with $H$ self-normalizing in $G$, Lemma \ref{lem:index} and Corollary \ref{cor:singular} tell us that $L(\mathscr{T}_d(H_*)) \subseteq L(\mathscr{T}_d(G_*))$ is a singular inclusion of type $\II_1$ factors which is finite index and hence the WAHP is precluded. We note that there are is abundance of such group inclusions. For example, let $n \ge 3$ and let $\text{stab}(i)$ be the subgroup of $S_n$ fixing $i \in \{1, \dots , n\}$. Then if $G = S_n$ and $H = \text{stab}(i)$, we have $H$ is self-normalizing in $G$, so Thompsonifying $H \le G$ gives us such an inclusion. For an example guaranteed to be non-amenable, let $G = \mathbb{F}_2 = \langle a,b \rangle$ and let $H$ be the preimage of the subgroup $\text{stab}(i)$ under the homomrphism $\mathbb{F}_2 \to S_n$ induced by $a \mapsto (1~2~ \dots ~n)$ and $b \mapsto (1~2)$. Then $H \le G$ is a finite index inclusion with $H$ self-normalizing in $G$, so Thompsonifying it gives us such an inclusion. Many other examples can be constructed in a similar fashion. 

As we noted, the examples constructed above do not satisfy the WAHP by virtue of the inclusion being finite index. We now construct singular inclusions of type $\II_1$ factors not satisfying the WAHP but which are of infinite index. To do this, we need a few more lemmas, which will enable us to construct a variety of such inclusions. 

\begin{lemma}\label{lem:notWAHP}
Let $H \le G$ be an inclusion of groups, and define $H_n$ and $G_n$ as above equipped with the $d$-ary cloning system described above. If $H \le G$ does not satisfy the right coset covering property, then $\mathscr{T}_d(H_*) \le \mathscr{T}_d(G_*)$ does not satisfy the right coset covering property. 
\end{lemma}
\begin{proof}
If $H \le G$ does not satisfy the coset covering property, then there exists $g \in G \setminus H$ and $y_1,...,y_t \in G$ such that 

$$Hg \subseteq \bigcup_{i=1}^{t} y_i H.$$

Note that $x := [\Lambda_d , (1,1, \dots ,1,g), \Lambda_d]$ is an element of $\mathscr{T}_d(G_*)$ which does not belong to $\mathscr{T}_d(H_*)$. Define $x_i = [\Lambda_d , (1,1,\dots,y_i), \Lambda_d]$ for $i=1,\dots,t$. We claim that

$$\mathscr{T}_d(H_*) x \subseteq \bigcup_{i=1}^{t} x_i \mathscr{T}_d(H_*),$$
which would prove that $\mathscr{T}_d(H_*) \le \mathscr{T}_d(G_*)$ does not have the right coset covering property.
To this end, let $[T,(1,g_2, \dots,g_{n-1},h),U]$ be an arbitrary element of $\mathscr{T}_d(H_*)$ for some $d$-ary trees $T$ and $U$ with $n$ leaves and some $h \in H$. Since $hg \in Hg$, there exists some $h' \in H$ such that $hg = y_ih'$ for some $i=1, \dots ,t$. Hence,

\begin{align*}
[T,(1,g_2,\dots,g_{n-1},h),U][\Lambda_d, (1,1,\dots,1,g) \Lambda_d] &= [T,(1,g_2,\dots,g_{n-1},h),U][U, (1,1,\dots,1 \underbrace{g,\dots,g,}g),U] \\
&= [T,(1,g_2 \epsilon_2 , g_3 \epsilon_3 ,\dots,g_{n-1} \epsilon_{n-1}, hg),U] \\
&= [T,(1,g_2 \epsilon_2 , g_3 \epsilon_3 ,\dots,g_{n-1} \epsilon_{n-1}, y_i h'),U] \\
&= [T,(1, \dots ,1, \underbrace{y_i,\dots,y_i,} y_i, T] \\
& [T, (1, \delta_2 g_2 \epsilon_2,\dots, \delta_{n-1} g_{n-1} \epsilon_{n-1},h'),U] \\
&= [\Lambda_d,(1,\dots,1,y_i),\Lambda_d] \\
&[T, (1, \delta_2 g_2 \epsilon_2,\dots, \delta_{n-1} g_{n-1} \epsilon_{n-1},h'),U] \\
&= x_i [T, (1, \delta_2 g_2 \epsilon_2,\dots, \delta_{n-1} g_{n-1} \epsilon_{n-1},h'),U]
\end{align*}
where $\epsilon_j \in \{1,g\}$ and $\delta_j \in \{1, y_i^{-1}\}$ for $j=2,\dots,n-1$, depending upon whether we need to clone the last coordinate in the process of expanding triples to build $T$ and $U$ from $\Lambda_d$. Note that some of the products had to be written on two lines. As in the proof of Lemma \ref{lem:index}, the underlined portions could possibly empty if we never need to clone the last coordinate in the process of building $T$ and $U$ from $\Lambda_d$ via expansions. In any case, note that this lies in $x_i \mathscr{T}_d(H_*)$, which completes the proof.
\end{proof}

Finally, we have an easy lemma concerning amgalmated free products. 

\begin{lemma}\label{lem:amalgamation}
Let $\Gamma_1$ be any finite group, $\Gamma_2$ any group, and $\Sigma$ any common proper subgroup. Then $\Gamma_1$ is self-normalizing in the amalgamated free product $\Gamma_1 \ast_{\Sigma} \Gamma_2$, and the inclusion $\Gamma_1 \le \Gamma_1 \ast_{\Sigma} \Gamma_2$ is an infinite index inclusion which does not satisfy the right coset covering property.
\end{lemma}
\begin{proof}
That $\Gamma_1$ is self-normalizing in $\Gamma_1 \ast_{\Sigma} \Gamma_2$ a fairly standard fact about amalgamated free products which can be gleaned from its action on its Bass-Serre tree or by considering normal forms of elements, and is independent of $\Gamma_1$ being finite. To see that the inclusion is of infinite index is also fairly easy to see: if $x \in \Gamma_1 \setminus \Sigma$ and $y \in \Gamma_2 \setminus \Sigma$ are any elements, then clearly $\{(xy)^n : n \in \mathbb{Z} \}$ are a set of distinct left coset representatives of $\Gamma_1$ in $\Gamma_1 \ast_{\Sigma} \Gamma_2$. Finally, the inclusion does not have the right coset covering property because, given $z \in \Gamma_1y$, where $y \in \Gamma_2 \setminus \Sigma$ is any element, we have $z = gy$ for some $g \in \Gamma_1$ or

$$z = gy = gy g^{\text{ord}(g)-1} g \in ghg^{\text{ord}(g)-1} \Gamma_1$$
which proves that 
$$\Gamma_1 y \subseteq \bigcup_{g \in \Gamma_1} gyg^{\text{ord}(g)-1} \Gamma_1,$$
where $\text{ord}(g)$ denotes the order of the group element $g$, which is necessarily finite because $\Gamma_1$ is a finite group. Since $y \notin \Gamma_1$ and the right coset $\Gamma_1 y$ is covered by a finite number of left cosets of $\Gamma_1$, the inclusion $\Gamma_1 \le \Gamma_1 \ast_{\Sigma} \Gamma_2$ does not have the right coset covering property. 
\end{proof}

For example, the inclusion $\mathbb{Z}_4 \le \mathbb{Z}_4  \ast_{\mathbb{Z}_2} \mathbb{Z}_6 \cong \text{SL}_2(\mathbb{Z})$ satisfies the conclusion of Lemma \ref{lem:amalgamation}, where $\text{SL}_2(\mathbb{Z})$ is the $2 \times 2$ special linear group over $\mathbb{Z}$. We also have the case when $A = \{e\}$, which reduces the amalgamated free product to the ordinary free product. In particular, we have that inclusions $\mathbb{Z}_2 \le \mathbb{Z}_2 \ast \mathbb{Z}_2 \cong D_{\infty}$ and $\mathbb{Z}_2 \le \mathbb{Z}_2 \ast \mathbb{Z}_3 \cong \text{PSL}_2(\mathbb{Z})$, where $D_{\infty}$ is the infinite dihedral group and $\text{PSL}_2(\mathbb{Z}) := \text{SL}_2(\mathbb{Z})/\{\pm I_2\}$ is the $2 \times 2$ projective special linear group with integer coefficients, or $\Gamma \le \Gamma \ast \mathbb{Z}$ and $\Gamma \le \Gamma \ast \mathbb{F}_2$ for any finite group $\Gamma$ satisfy the hypotheses of Lemma \ref{lem:amalgamation}.

How do these lemmas and the corollary help us in constructing a singular, infinite index inclusions of type $\II_1$ factors without the WAHP? Using Lemma \ref{lem:index}, Corollary \ref{cor:singular}, Lemma \ref{lem:notWAHP}, and Lemma \ref{lem:amalgamation}, we can construct a variety of examples, including the ones mentioned in the preceding paragraph. Let $\Gamma_1$, $\Gamma_2$, and $\Sigma$ satisfy the hypotheses of Lemma \ref{lem:amalgamation}, and define $H = \Gamma_1$ and $G = \Gamma_1 \ast_{\Sigma} \Gamma_2$. Then by Lemma \ref{lem:index}, Lemma \ref{lem:normalizer}, Lemma \ref{lem:notWAHP}, and Lemma \ref{lem:amalgamation}, we see that $\mathscr{T}_d(H_*)$ is self-normalizing in $\mathscr{T}_d(G_*)$ and the inclusion $\mathscr{T}_d(H_*) \le \mathscr{T}_d(G_*)$ is an infinite index inclusion which does not satisfy the right coset covering condition. Because not satisfying the right coset covering property is equivalent to not satisfying the WAHP, when we translate the group inclusion to an inclusion of their group von Neumann algebras, Corollary \ref{cor:singular} and Lemma \ref{lem:notWAHP} tell us that $L(\mathscr{T}_d(H_*)) \subseteq L(\mathscr{T}_d(G_*))$ is an infinite index inclusion of $\II_1$ factors which is singular but does not satisfy the WAHP. Although somewhat besides the point, we note that the factors considered in this section, both $L(\mathscr{T}_d(H_*))$ and $L(\mathscr{T}_d(G_*))$, are McDuff.

One might wonder why we did not use the inclusion $L(F_d) \subseteq L(\mathscr{T}_d(G_*))$ to construct examples of such inclusions. The reason for this is that  that so many of the $d$-ary cloning systems are diverse, so it seems unlikely, and indeed would be surprising, that there exists a sequence of groups $(G_n)_{n \in \mathbb{N}}$ equipped with a $d$-ary cloning system such that $L(F_d) \subseteq L(\mathscr{T}_d(G_*))$ is a singular inclusion of type $\II_1$ factors without the WAHP. Necessarily, the $d$-ary cloning system would have to be non-diverse. But we saw that even non-diverse $d$-ary cloning systems give rise to inclusions satisfying the WAHP, or if it does not have the WAHP, it did not have it because it was non-singular or not irreducible. We leave it to future work to determine if this is possible.

\section{The Higman--Thompson Groups $F_d$ are McDuff}\label{sec:higmanmcduff}

In this section, we use Theorem \ref{thrm:Irreducible}, Citation~\ref{mcduff}, and character rigidity of the Higman--Thompson groups $F_d$ to deduce that the $F_d$ are McDuff groups for all $d \ge 2$. First, let us recall the notion of a McDuff group, first singled out by Deprez and Vaes in \cite{deprez18}. A group $G$ is said to be \emph{McDuff} provided it admits a free, ergodic, probability measure preserving (p.m.p.) action $\sigma$ on a standard probability measure space $(X, \mu)$ such that the corresponding crossed product von Neumann algebra $L^{\infty}(X,\mu) \rtimes_{\sigma} G$ is a type $\II_1$ McDuff factor. It turns out that a group $G$ being McDuff in the sense of Deprez-Vaes implies that $G$ is inner amenable (see \cite{deprez18} and \cite{choda}). Hence for a group $G$, if either $L(G)$ is a type $\II_1$ McDuff factor, or $G$ is a McDuff group in the above sense, $G$ is inner amenable. These two properties, therefore, bear the same relationship to inner amenability, but the relationship between them is not entirely clear as far as we can tell. Part of this relationship was already clarified when Kida constructed an ICC group $G$ that is McDuff in the above sense yet is such that $L(G)$ does not have property Gamma and therefore cannot be a McDuff factor (see \cite{kida15}). But as far as we can tell, it is unknown whether $L(G)$ being a type $\II_1$ McDuff factor implies $G$ is a McDuff group. 

Let us recall what it means for a group $G$ to be character rigid. For an excellent reference on character rigidity, we point the interested reader to the survey \cite{peterson2016lecture}. Recall that $\tau : G \to \mathbb{C}$ is a \textit{character} in the operator-algebraic or representation-theoretic sense if
\begin{enumerate}[1.]
    \item $\tau (e) = 1$, \\
    \item $[\tau (g_j^{-1} g_i)]_{1 \le i,j \le n}$ is a non-negative definite $n \times n$ matrix for all $g_1,\dots,g_n \in G$ and $n \in \mathbb{N}$, and \\
    \item $\tau (g_1g_2) = \tau (g_2 g_1)$ for all $g_1,g_2 \in G$.
\end{enumerate} 
The space of characters on $G$ forms a $\text{weak}^*$-closed convex subset of $\ell^{\infty}(G)$. A character $\tau : G \to \mathbb{C}$ is said to be \emph{almost periodic} if the set of translates $\{x \mapsto \tau (gx)\}_{g \in G}$ is uniformly precompact in $\ell^{\infty}(G)$. With this in mind, a group $G$ is said to be \emph{character rigid} if whenever a character $\tau : G \to \mathbb{C}$ is an extreme point in the space of characters on $G$, either $\tau$ is almost periodic or $\tau = \delta_e$. A key dynamical property of character rigid groups is that any ergodic probability measure preserving action on a standard probability measure space is necessarily also free, and it is this property, in particular, we will exploit to prove the groups $F_d$ are McDuff in the above sense. 

Via the Gelfand-Naimark-Segal (GNS) construction, characters of a group $G$ are in one-to-one correspondence with unitary representation of $G$ into tracial von Neumann algebras, and a character is an extreme point in the space of characters if and only if the corresponding tracial von Neumann algebra is additionally a factor (see \cite[Theorem~5.7]{peterson2016lecture}). Given a unitary representation $\pi : G \to \mathcal{U}(\h)$ of $G$, where $\h$ is some separable Hilbert space, $\pi$ is said to be a \emph{finite factor representation} if the von Neumann algebra generated by $\pi(G)$ is a finite factor von Neumann algebra (finite factors coincide with tracial von Neumann algebras). In \cite{dudko}, Dudko and Medynets characterized all finite factor representations of ``extended" Higman--Thompson groups $\{F_{d,r}\}_{d,r \in \mathbb{N}, d \ge 2}$ and $\{V_{d,r}\}_{d,r \in \mathbb{N}, d \ge 2}$, and from this characterization it follows that the Higman--Thompson groups $\{F_{d,r}\}_{d,r \in \mathbb{N}, d \ge 2}$ and $\{V_{d,r}\}_{d,r \in \mathbb{N}, d \ge 2}$ are character rigid (in their notation, $G_{d,r} = V_{d,r}$). This entails, of course, they have the aforementioned dynamical property---namely, if $G \in \{F_{d,r} \}_{d,r \in \mathbb{N}, d \ge 2} \cup \{V_{d,r}\}_{d,r \in \mathbb{N}, d \ge 2}$ admits a non-trivial, ergodic action on a standard probability measure space, then the action is necessarily free. In particular, these results hold for $F_d = F_{d,1}$. This key  dynamical property, combined with Citation \ref{mcduff} and Theorem \ref{thrm:Irreducible}, can be used to prove the Higman--Thompson groups $F_d$ are McDuff.

\begin{theorem}\label{thrm:higmanmcduff}
The Higman--Thompson group $F_d$ is McDuff for all $d \ge 2$.
\end{theorem}
\begin{proof}
First, let $(G_n)_{n \in \mathbb{N}}$ be any sequence of non-trivial abelian groups equipped with a pure, diverse, and uniform $d$-ary cloning system. Because the $d$-ary cloning system is pure, we get the internal semi-direct product decomposition $\mathscr{T}_d(G_*) = \mathscr{K}_d(G_*) \rtimes F_d$ which translates to a crossed product decomposition when we pass to its group von Neumann algebra. Because the $d$-ary cloning system is pure, uniform, and diverse, we know by Citation \ref{mcduff} that $L(\mathscr{T}_d(G_*))$ is a type $\II_1$ McDuff factor.

Moreover, we know that the action of $F_d$ on $\widehat{\mathscr{K}_d(G_*)}$ is probability measure preserving. This is because $F_d$ acts on $\mathscr{K}_d(G_*)$ via conjugation which translates to the induced action on $L(\mathscr{K}_d(G_*))$ being trace-preserving which translates to the action on $L^{\infty}(\mathscr{K}_d(G_*))$ being integral-preserving and, finally, this translates to the action of $F_d$ on $\widehat{\mathscr{K}_d(G_*)}$ being probability measure preserving. 

Because the $d$-ary cloning system is diverse, we know by Theorem \ref{thrm:Irreducible} that $L(F_d)$ is an irreducible subfactor of the type $\II_1$ McDuff factor $$L(\mathscr{T}_d(G_*)) \cong L(\mathscr{K}_d(G_*)) \rtimes F_d \cong L^{\infty}(\widehat{\mathscr{K}_d(G_*)}) \rtimes F_d.$$ Since $F_d$ is an ICC group, irreducibility is equivalent to $F_d$ acting ergodically on $L^{\infty}(\widehat{\mathscr{K}_d(G_*)})$ and hence on $\widehat{\mathscr{K}_d(G_*)}$. By virtue of character rigidity of $F_d$, we know that this probability measure preserving ergodic action on $\widehat{\mathscr{K}_d(G_*)}$ must also be free. Hence, putting all of this together, we have constructed a standard probability measure space on which $F_d$ admits a free, ergodic, probability measure preserving action such that the corresponding crossed product $$L^{\infty}(\widehat{\mathscr{K}_d(G_*)}) \rtimes F_d$$
is a McDuff type $\II_1$ factor. Whence, it follows $F_d$ is a McDuff group.
 \end{proof}

For an example of the $d$-ary cloning system used in the proof of the above theorem to witness McDuffness of $F_d$, consider $G_n := \Psi^n(G)$ with $G$ any non-trivial abelian group (even as simple as $G = \mathbb{Z}_2$) and the monomorphisms $\phi_1, \dots , \phi_d$ all equal to the identity. This $d$-ary cloning system is pure, diverse, and uniform so by Citation~\ref{mcduff} we know that $L(\mathscr{T}_d(G_*))$ is a type $\II_1$ McDuff factor. Although the Higman--Thompson groups $F_d$ are McDuff, we note that the other Higman--Thompson groups $T_d$ and $V_d$ cannot be McDuff since they are not inner amenable (see \cite{bashwinger1}).

From one perspective, this approach to proving $F_d$ is McDuff is especially nice because it obviates having to explicitly construct a probability measure space on which $F_d$ admits a free, ergodic, and p.m.p. action, and then checking whether the corresponding crossed product is a type $\II_1$ McDuff factor, all of which can be difficult in general. It is noteworthy that none of the natural spaces on which $F_d$ acts (the unit interval $[0,1]$, the unit circle $S^1$, the $d$-ary Cantor space $\mathcal{C}_d$, etc.) are p.m.p actions. Moreover, this approach to proving $F_d$ is a McDuff group uses the same criteria used to show that $L(F_d)$, in addition to other Thompson-like group factors, is a McDuff.

From another perspective, however, these can be seen as disadvantages of the proof: it is in a sense non-constructive, bordering on abstract nonsense; it does not give us very much insight into proving a given group is McDuff; and it does very little in the way of illuminating the relationship between an ICC group $G$ being McDuff in the sense of Deprez-Vaes and $L(G)$ being a type $\II_1$ McDuff factor. In addition, we suspect that other Thompson-like groups, particularly those arising from (slightly) pure $d$-ary cloning systems, should yield McDuff groups, but the above proof does not obviously generalize. With that being said, the proof does does establish that $F_d$ is McDuff and lends some evidence to other Thompson-like groups possibly being McDuff, and we leave it to future work to explore this.    

We also remark that the techniques used in the proof of Theorem~\ref{thrm:higmanmcduff} can be used to show that $L(\mathscr{K}_d(G_*))$ can give rise to a Cartan subalgebra in $L(\mathscr{T}_d(G_*))$. Recall that a Cartan subalgebra in a von Neumann algebra is a regular maximal abelian subalgebra. More precisely, using those techniques we can prove that if $(G_n)_{n \in \mathbb{N}}$ is a sequence of non-trivial abelian groups equipped with a pure and diverse $d$-ary cloning system (the uniform assumption is superfluous in this case), then $L(\mathscr{K}_d(G_*))$ is a Cartan subalgebra of $L(\mathscr{T}_d(G_*))$. Indeed, from the diversity assumption we know that the inclusion $$L(F_d) \subseteq L(\mathscr{T}_d(G_*)) \cong L(\mathscr{K}_d(G_*)) \rtimes F_d$$ is irreducible by Theorem \ref{thrm:Irreducible} which translates to the action of $F_d$ on the abelian von Neumann algebra $L(\mathscr{K}_d(G_*))$ being ergodic and hence free by character rigidity. For crossed product decompositions like this, it is a standard fact that the action being free is equivalent to $L(\mathscr{K}_d(G_*))$ being a Cartan subalgebra of $L(\mathscr{K}_d(G_*)) \rtimes F_d \cong L(\mathscr{T}_d(G_*))$. Hence, we have the following theorem.

\begin{theorem}\label{cartan}
Let $((\rho_n)_{n \in \mathbb{N}}, (\kappa_k^n)_{k \le n}))$ be a pure and diverse $d$-ary cloning system on a sequence of non-trivial abelian groups $(G_n)_{n \in \mathbb{N}}$. Then $L(\mathscr{K}_d(G_*))$ is a Cartan subalgebra of $L(\mathscr{T}_d(G_*))$.
\end{theorem}
What makes this theorem intriguing is that it starkly contrasts with the case of $L(F_d)$, $L(T_d)$, and $L(V_d)$. Indeed, as we noted in the introduction, the Higman-Thompson group factors $L(F_d)$, $L(T_d)$, and $L(V_d)$ cannot contain a Cartan subalgebra arising from an abelian subgroup\footnote{ We note, however, that $L(\langle x_0 \rangle)$ is a singular maximal abelian subalgebra of $L(F)$ (see \cite[Lemma 3.2]{jolissaint05}), where $x_0$ is the first generator in the infinite presentation of $F$, and this ought to easily generalize to $F_d$ for all $d > 2$.}, and the same is true of the R{\"o}ver--Nekrashevych groups because any normal subgroup of a R{\"o}ver--Nekrashevych group contains its commutator subgroup (see \cite[Theorem 9.11]{nekrashevych04}). However, it is rather easy to produce Thompson-like groups arising from $d$-ary cloning systems which contain an abelian subgroup giving rise to a Cartan subalgebra. As a matter of fact, given any non-trivial abelian group $G$ and any monomorphisms $\phi_1,\dots,\phi_d : G \to G$, we know that the usual $d$-ary cloning system on the sequence of groups $(\Psi^n(G))_{n \in \mathbb{N}}$ is pure and diverse so the above implies that $L(\mathscr{K}_d(\Psi^*(G)))$ is always a Cartan subalgebra of the type $\II_1$ factor 
$L(\mathscr{T}_d(\Psi^*(G)))$. If we additionally assume that the $\phi_1,\dots,\phi_d$ are the identity, then $L(\mathscr{T}_d(\Psi^*(G)))$ is in fact a type $\II_1$ McDuff factor with at least one Cartan subalgebra for any choice of non-trivial abelian group $G$. Type $\II_1$ McDuff factors with at least one Cartan subalgebra are interesting because their Cartan subalgebras are not classifiable by countable structures in the descriptive set-theoretic sense (see \cite[Corollary G]{spaas18}). Since it is entirely possible that $F_d$ is amenable, these factors $L(\mathscr{T}_d(\Psi^*(G)))$ would merely be different manifestations of the hyperfinite type $\II_1$ factor. Hence, it would be interesting to find a non-amenable group $G$, in order to ensure that $\mathscr{T}_d(\Psi^*(G))$ is non-amenable, such that $L(\mathscr{T}_d(\Psi^*(G))$ contains at least one Cartan subalgebra

At the moment, it is unclear how to generalize Theorem~\ref{cartan}. We note that it certainly will not be true for all fully compatible $d$-ary cloning systems. To see why this is the case, recall that the Higman--Thompson group $T_d$ arises from the standard $d$-ary cloning system on $(S_n)_{n \in \mathbb{N}}$ restricted to the non-trivial abelian groups $(\langle (1~2~\cdots~n ) \rangle)_{n \in \mathbb{N}}$. This $d$-ary cloning system is fully compatible and diverse. However, as we have already emphasized, $T_d$ contains no infinite normal abelian subgroups so $L(\mathscr{K}_d(G_*))$ certainly will not be a Cartan subalgebra. Thus, it is not clear how to generalize Theorem~\ref{cartan} from pure $d$-ary to a larger class of $d$-ary cloning systems not including $T_d$, or whether it can be generalized at all. We leave it to future work to determine this as well as studying, more generally, maximal abelian subalgebras in group von Neumann algebras of Thompson-like groups arising from $d$-ary cloning systems.

\bibliographystyle{alpha}
\bibliography{bibliography}

\end{document}